\documentclass[english]{amsart}
\usepackage[T1]{fontenc}
\usepackage[latin9]{inputenc}
\usepackage{amsmath}
\usepackage{amssymb}
\usepackage{amsmath}
\usepackage{amstext}
\usepackage{amsfonts}
\usepackage{amscd}
\usepackage{amsbsy}
\usepackage{amsthm}
\usepackage{ifthen}
\usepackage{enumerate}
\usepackage{latexsym}
\usepackage{babel}
\usepackage{xy}
\usepackage{graphics}
\xyoption{all}

\makeatletter

\makeatletter
\def\input@path{{/home/filiuspelei/Hope//}}
\makeatother

\newcounter{dummy} \numberwithin{dummy}{section}
\newtheorem*{theorem*}{Theorem}

\newtheorem{theorem}[dummy]{Theorem}
\newtheorem{prop}[dummy]{Proposition}
\newtheorem{lemma}[dummy]{Lemma}

\newtheorem{corollary}[dummy]{Corollary}
\newtheorem{conjecture}[dummy]{Conjecture}

\theoremstyle{definition}

\newtheorem{definition}[dummy]{Definition}
\newtheorem{example}[dummy]{Example}

\newtheorem{remark}[dummy]{Remark}

\providecommand{\tr}{\mathop{\rm Tr}\nolimits}
\providecommand{\im}{\mathop{\rm Im}\nolimits}
\providecommand{\h}{\mathop{\rm ht}\nolimits}
\providecommand{\coker}{\mathop{\rm coker}\nolimits}

\providecommand{\spec}{\mathop{\rm spec}\nolimits}

\providecommand{\hm}{\mathop{\rm Hom}\nolimits}
\providecommand{\ext}{\mathop{\rm Ext}\nolimits}
\providecommand{\tor}{\mathop{\rm Tor}\nolimits}
\providecommand{\mor}{\mathop{\rm Morph}\nolimits}

\providecommand{\md}{\mathop{\rm mod}\nolimits}
\providecommand{\id}{\mathop{\rm \textbf{id}}\nolimits}
\providecommand{\res}{\mathop{\rm res}\nolimits}
\providecommand{\grade}{\mathop{\rm grade}\nolimits}
\providecommand{\zz}{\Omega}
\providecommand{\pd}{\mathop{\rm pd}\nolimits}

\providecommand{\depth}{\mathop{\rm depth}\nolimits}

\providecommand{\add}{\mathop{\rm add}\nolimits}
\providecommand{\DD}{\mathop{\rm \Delta}\nolimits}

\providecommand{\Md}{\mathop{\rm Mod}\nolimits}

\providecommand{\V}{\mathrm{V}}
\providecommand{\NA}{\mathop{\rm NA}\nolimits}
\providecommand{\Ares}{\res_\mathcal{A}}
\providecommand{\atr}{\tr_\mathcal{A}}
\providecommand{\Trc}{\tr_C}

\providecommand{\thick}{\mathop{\rm Thick}\nolimits}
\providecommand{\hdim}{\mathop{\rm \mbox{-}dim}\nolimits}
\providecommand{\ABdim}{\mathop{\rm AB\mbox{-}dim}\nolimits}
\providecommand{\hmor}{\mathop{\rm H\mbox{-}Morph}\nolimits}

\providecommand{\NN}{\mathbb{N}}

\providecommand{\A}{\mathcal{A}}
\providecommand{\B}{\mathcal{B}}
\providecommand{\Y}{\mathcal{Y}}
\providecommand{\X}{\mathcal{X}}
\providecommand{\Z}{\mathcal{Z}}

\providecommand{\PP}{\mathcal{P}}
\providecommand{\C}{\mathcal{C}}
\providecommand{\N}{\mathcal{N}}

\providecommand{\M}{\mathcal{M}}
\providecommand{\G}{\mathcal{G}}
\providecommand{\W}{\mathcal{W}}

\providecommand{\Ss}{\mathfrak{S}}
\providecommand{\R}{\mathfrak{R}}

\providecommand{\mm}{\mathfrak{m}}

\providecommand{\p}{\mathfrak{p}}

\providecommand{\Ga}{\Gamma}
\providecommand{\de}{\delta}

\providecommand{\ep}{\varepsilon}

\providecommand{\ta}{\theta}
\providecommand{\io}{\iota}

\providecommand{\ld}{\lambda}
\providecommand{\Ld}{\Lambda}

\providecommand{\ph}{\varphi}

\providecommand{\sbe}{\subseteq}

\providecommand{\sbne}{\subsetneq}

\providecommand{\x}{\times}
\providecommand{\xto}{\xrightarrow}

\providecommand{\del}{\backslash}

\providecommand{\jn}{\lor}
\providecommand{\da}{\dagger}
\providecommand{\op}{\oplus}
\providecommand{\DA}{\DD(\mathcal{A})}

\providecommand{\dell}{\partial}

\providecommand{\MCM}{\mbox{MCM}}
\providecommand{\GDZ}{\mbox{GDZ}}

\title{Classifying resolving subcategories}

\author{William Sanders}

\begin{document}
\maketitle
\begin{abstract}
We use the theory of Auslander Buchweitz approximations to classify certain resolving subcategories containing a semidualizing or a dualizing module.  In particular, we show that if the ring has a dualizing module, then the resolving subcategories containing maximal Cohen-Macaulay modules are in bijection with grade consistent functions and thus are the precisely the dominant resolving subcategories.  
\end{abstract}

\section{Introduction}  Classifying various types of subcategories  of $\md(R)$ and $D(R)$ for a commutative ring $R$ has been the subject of much recent research.  These classifications are intrinsically connected to $\spec R$ or some other topological space.  
For instance, the Hopkins Neeman Theorem in \cite{Hopkins87} and \cite{Neeman92} and Gabriel's Theorem in \cite{Gabriel62}  give  a bijection between the Serre subcategories of $\md(R)$, the thick subcategories of perfect complexes, and the specialization closed subsets of $\spec R$.  Another example is the work regarding the classification of thick subcategories of $\md R$ such as in \cite{Takahashi10} and \cite{Stevenson12}.  

Recently, much attention has been given to classifying the resolving subcategories of $\md(R)$.  The study of resolving subcategories began with  Auslander and Bridger's influential work  in \cite{AuslanderBridger69} where they define the category of Gorenstein dimension zero modules, which we will denote by $\GDZ$.  Also, they  generalize the notion of projective dimension by defining Gorenstein dimension through approximations of Gorenstein dimension zero modules.  In their paper, they also prove that  $\GDZ$ has certain homological closure properties which cause Gorenstein dimension to behave similarly to projective dimension.  They then take these homological closure properties of $\GDZ$ as the definition of resolving subcategories.  We can take dimension with respect to a resolving subcategory, and, as in the case of $\GDZ$, these homological closure properties force this dimension function to also behave similarly to projective dimension.  See Section \ref{alpha} for further exposition. 

The classification of resolving subcategories was advanced by Dao and Takahashi in \cite{DaoTakahashi13}, where they give a bijection between the  set of resolving subcategories of the category of finite projective dimension modules and the set of grade consistent functions.  A function $f:\spec R\to \NN$ is called grade consistent if it is increasing (as a morphism of posets) and $f(p)\le \grade(p)$ for all $p\in\spec(R)$. This result motivated the author to find other situations where a similar bijection exists, furthering the use of grade consistent functions in classifying resolving subcategories.    Before the work of Dao and Takahashi, Takahashi classifies, over Cohen-Macaulay rings, resolving subcategories  closed under tensor products and Auslander transposes in \cite{Takahashi13}, and in \cite{Takahashi11} he classifies the contravariantly finite resolving subcategories of a Henselian local Gorenstein ring.   In \cite{Takahashi09}, Takahashi also studies  resolving subcategories which are free on the punctured spectrum.  In \cite{AuslanderReiten91}, Auslander and Reiten discover a connection between resolving subcategories and tilting theory, and they classify all the contravariantly finite resolving subcategories using cotilting bundles.  After the work of Dao and Takahashi, the resolving  subcategories of the category of finite projective dimension modules were also classified in \cite{HugelPospisiletal12} in terms of descending sequences of specialization closed subsets of $\spec R$, and were also classified in \cite{HugelSaorin14} in terms of certain t-structures.  

In this paper, we assume that $R$ is commutative and Noetherian, and we consider only finitely generated modules. Let $\PP$ denote the category of projective modules and $\Ga$ the set of grade consistent functions.  For categories $\M,\X\sbe\md(R)$ and $f\in\Ga$, we define 
\[\Ld_\M(f)=\{X\in\md(R)\mid \add \M_p\hdim X_p\le f(p)\quad  \forall p\in\spec R\}\]
\[\Phi_\M(\X):\spec R\to \NN\quad\quad\quad p\mapsto \sup\{\add \M_p\hdim X_p\mid X\in\X\}\]
where $\add \M_p$ is the smallest subcategory of $\md(R_p)$ closed under direct sums and summands and containing $M_p$ for every $M\in \M$, and $\add \M_p\hdim X_p$ is the smallest resolution of $X_p$ by objects in $\add \M_p$.  Let $\R$ denote the collection of resolving subcategories of $\md(R)$.  For any $\M\sbe\md(R)$, set $\DD(\M)=\{X\in\md(R)\mid \M\hdim X<\infty\}$ and let $\R(\M)$ be the collection of resolving subcategories $\X$ such that $\M\sbe \X\sbe\DD(\M)$.   Using our new notation, we can restate  Dao and Takahashi's result from \cite{DaoTakahashi13}.

\begin{theorem}
When $R$ is Noetherian, the following is a bijection
\[\xymatrix{
\R(\PP) \ar@<.5ex>[r]^{\Ld_\PP} 	& \Ga \ar@<.5ex>[l]^{\Phi_\PP}\\
}\]
where $\Ld_\PP$ and $\Phi_\PP$ are inverses of each other.
\end{theorem} 
Our first main result is Theorem \ref{engine}, which is the following.  Note that throughout this paper, all thick subcategories contain $R$.

\begin{theorem*}[A]
Let $\Psi$ be a set of increasing functions from $\spec R$ to $\NN$.  Suppose $\A\sbe \M$ such that $\A$ cogenerates $\M$ and $\add \A_p$ is thick in $\add \M_\p$ for all $\p\in\spec R$.  Define $\eta_\A^\M:\R(\A)\to \R(\M)$ by $\eta_\A^\M(\X)=\res (\X\cup\M)$ and $\rho_\A^\M:\R(\M)\to \R(\A)$ by setting $\rho_\A^\M(\X)=\DD(\A)\cap \X$.  If $\Phi_\A$ and $\Ld_\A$ are inverses of each other giving a bijection between $\R(\A)$ and $\Psi$, then we have the following commutative diagram
\[\xymatrix{
\R(\M) \ar[dr]^{\Phi_\M} & &\\
& \Psi &\\
\R(\A) \ar[ur]^{\Phi_\A} \ar[uu]^{\eta_\A^\M}& &\\
}\]
where $\Phi_\M$ is bijective with $\Ld_\M$ its inverse.   Moreover,  $\rho_\A^\M$ is the inverse of $\eta_\A^\M$.  
\end{theorem*}

This result allows us to extend the bijection from \cite{DaoTakahashi13} to a plethora of categories.  We use it to prove the following result which is essentially Theorem \ref{main3}. Note that $\G_C$ is the category of totally $C$-reflexive modules where $C$ is a semidualizing module: see Definition \ref{tot} and Definition \ref{sdm}.  Define ,  $\rho_\M^\N$ and  $\eta_\M^\N$ similarly to $\rho_\A^\M$ and $\eta_\A^\M$.  
 
\begin{theorem*}[B]
For any  thick subcategory $\M$ of $\G_C$ containing $C$, $\Ld_\M$ and $\Phi_\M$ give a bijection between $\R(\M)$ and $\Ga$. 

Furthermore, let $\Ss$ denote the collection of thick subcategories of  $\G_C$ containing $C$.  The following is a bijection.
 \[\Ld:\Ss\x\Ga\to \bigcup_{\M\in\Ss} \R(\M)\sbe \R\] 
For any $\M,\N\in\Ss$ with $\M\sbe \N$, then the following diagram commutes.
\[\xymatrix{
\R(\N) \ar[dr]^{\Phi_\N} 							& 		&\\
					& \Ga &\\
\R(\M) \ar[ur]^{\Phi_\M} \ar[uu]^{\eta_\M^\N}	& 		&\\
}\]
In particular,  $\rho_\M^\N$ and  $\eta_\M^\N$ are inverse functions. 

\end{theorem*}

These theorems show that the classification of resolving subcategories is intrinsically linked to the classification of thick subcategories of totally $C$-reflexive modules and hence to the classification of thick subcategories of $\md(R)$, a topic of current research. See for instance \cite{Takahashi10} or \cite{Neeman92}.  Applying these results in the Gorenstein case yields Theorem \ref{maingor} which, letting $\MCM$ denote the category of maximal Cohen-Macaulay modules, states

\begin{theorem*}[C]
If $R$ is Gorenstein, then we have the following bijections which commute
\[\xymatrix{
\{\mbox{Thick subcategories of }\MCM\}\x \Ga \ar[dr]^{\Ld} \ar[dd]^{\Ld_\PP}& &\\
& \{\Z\in \R\mid \Z\cap \MCM \mbox{ is thick in } \MCM\} &\\
\{\mbox{Thick subcategories of }\MCM\}\x \R(\PP) \ar[ur]^{\Xi} & &\\
}\]where $\Xi(\M,\X)=\res(\M\cup\X)$.  

\end{theorem*} 

Of independent interest, using semidualizing modules, we generalize the famed Auslander transpose.  This generalization is similar but different to the generalizations of Geng and Huang in \cite{Geng13} and \cite{Huang99}.    

This paper is organized as follows:  Section \ref{alpha} gives general information about resolving subcategories, and Section \ref{beta} gives pertinent background regarding semidualizing modules.  We prove Theorem A in Section \ref{gamma}.  In Section \ref{delta'}, we generalize the Auslander transpose, which we use in Section \ref{delta} to  classify resolving subcategories which are locally Maximal Cohen-Macaulay.  In Section \ref{epsilon} we prove a special case  of Theorem B.   We prove Theorem B in full generality in Section \ref{zeta} by examining the thick subcategories of maximal Cohen-Macaulay modules containing $C$.  In the last section, these results are applied to the Gorenstein case.  Here, Theorem C and several other results are proven.

\section{Resolving Preliminaries}\label{alpha}
We proceed with an overview of resolving subcategories.  All subcategories considered are full and closed under isomorphisms.  For any collection $\M\sbe \md(R)$, let $\add(\M)$ be the smallest subcategory of $\md(R)$ containing $\M$ which is closed under direct sums and summands.  
\begin{definition}
Given a ring $R$, a full subcategory $\M\sbe\md(R)$ is resolving if the following hold.
\begin{enumerate}
\item $R$ is in $\M$
\item $M\op N$ is in $\M$ if any only if $M$ and $N$ are in $\M$
\item If $0\to M\to N\to L\to 0$ is exact and  $L\in\M$, then $N\in \M$ if and only if $M\in\M$.
\end{enumerate}
\end{definition}
By \cite[Lemma 3.2]{Yoshino05}, this is equivalent to saying these conditions hold.
\begin{enumerate}
\item All projectives are in $\M$
\item If $M\in \M$, then $\add(M)\sbe \M$
\item $\M$ is closed under extensions
\item  $\M$ is closed under syzygies
\end{enumerate}
For a subset $\M\sbe\md(R)$, we denote by $\res(\M)$ the smallest resolving subcategory containing $\M$.  Also, $\add \M$ will    Let $\PP$ be the category of finitely generated projective $R$-modules.

\begin{example}\label{example}
The following categories are easily seen to be resolving.
\begin{enumerate}
\item $\PP$
\item $\md(R)$
\item The set of Gorenstein dimension zero modules
\item For any $\B\sbe \Md(R)$ and any $n\ge 0$, $\{M\mid\ext^{>n}(M,B)=0\  \forall B\in \B\}$
\item For any $\B\sbe \Md(R)$ and any $n\ge 0$, $\{M\mid\tor^{>n}(M,B)=0\  \forall B\in \B\}$
\item When $R$ is Cohen-Macaulay, the set of maximal Cohen-Macaulay modules
\end{enumerate}
\end{example}

A special class of resolving subcategories are thick subcategories.
\begin{definition}
Let $\N\sbe\md(R)$.  A resolving subcategory $\M\sbe\N$ is a thick subcategory of $\N$ (or $\M$ is thick in $\N$)  if for any exact sequence $0\to L\to M\to N\to 0$ with $L,M\in\M$, if $N$ is in $\N$, then $N$ is in $\M$ too.  A thick subcategory refers to a thick subcategory of $\md(R)$. 
\end{definition}

For any $\M\sbe\md(R)$, let $\thick(\M)$ be the smallest thick subcategory of $\md(R)$ containing $\M$.  
\begin{example}
The following categories are easily seen to be thick subcategories.  Moreover, each example is the thick closure of a resolving subcategory in Example \ref{example}.
\begin{enumerate}
\item The set of  modules with  finite projective dimension 
\item $\md(R)$
\item The set of modules with finite Gorenstein dimension
\item For any $\B\sbe \Md(R)$ and any $n\ge 0$, $\{M\mid\ext^{\gg 0}(M,B)=0\  \forall B\in \B\}$
\item For any $\B\sbe \Md(R)$ and any $n\ge 0$, $\{M\mid\tor^{\gg 0}(M,B)=0\  \forall B\in \B\}$
\end{enumerate}
\end{example}
Resolving subcategories are studied in part because dimension with respect to a resolving subcategory has nice properties.  For a subset $\M\sbe\md(R)$ and a module $X\in \md(R)$, we say that $\M\hdim X=n$ if $n\in\NN$ is the smallest number such that there is an exact sequence 
\[0\to M_n\to \cdots\to M_0\to X\to 0\]
with $M_0,\dots,M_n\in\M$.  Projective dimension and Gorenstein dimension are dimensions with respect to resolving subcategories of projective modules and Gorenstein dimension zero modules respectively.  The following proposition from \cite[Proposition 3.3]{AuslanderBuchweitz89}  causes nice properties to hold for dimension with respect to a resolving subcategory.

\begin{prop}\label{dim}
If $\M$ is resolving and  $\M\hdim(X)\le n$, then for any exact sequence 
\[0\to L\to M_{n-1}\to\cdots\to M_0\to X\to 0\]
with each $M_i\in\M$,  $L$ is in $\M$.   
\end{prop}
This proposition allows us to prove the following results.
\begin{corollary}\label{1.1}
If $\M$ is resolving, then $\M\hdim(X)=\inf\{n\mid \zz^n X\in\M\}$.
\end{corollary}
\begin{proof}
If $\zz^n X\in\M$, then we have
\[0\to \zz^n X\to F_{n-1}\to \cdots\to F_0\to X\to 0\]
with each $F_i$ projective.  This shows that $\M\hdim X\le n$.  If $\M\hdim X\le n$, the same sequence and Proposition \ref{1.1} show that $\zz^n X$ is in $\M$.
\end{proof}
\begin{lemma}
If $\M$ is resolving, then $\M\hdim X\op Y=\max\{\M\hdim X,\M\hdim Y\}$.
\end{lemma}
\begin{proof}
We have $\zz^n (X\op Y)=\zz^n X\op \zz^n Y$ for  a suitable choice of syzygies.  Since $\zz^n (X\op Y)$ is in $\M$ if and only if $\zz^n X$ and $\zz^n Y$ are in $\M$, the result follows from Corollary \ref{1.1}. Parts (\ref{21}) and (\ref{22}) are essentially proved in \cite[Theorem 18]{Masek99}.
\end{proof}
\begin{lemma}
If $\M$ is a resolving subcategory, and $0\to X\to Y\to Z\to 0$ is exact, then the following inequalities hold.
\begin{enumerate}
\item\label{21} $\M\hdim X\le\max\{\M\hdim Y,\M\hdim Z-1\}$
\item\label{22}  $\M\hdim Y\le \max\{\M\hdim X,\M\hdim Z\}$
\item\label{23}  $\M\hdim Z\le\max\{\M\hdim X,\M\hdim Y\}+1$
\end{enumerate}
\end{lemma}
\begin{proof}
For suitable choices of syzygies, we have the following.
\[0\to \zz^k X\to \zz^k Y\to \zz^k Z\to 0\]
If $k= \max\{\M\hdim X,\M\hdim Z\}$, then,  by Corollary \ref{1.1}, $\zz^k X$ and $\zz^k Z$ are in $\M$, and thus, so is $\zz^k Y$, giving us \ref{22}.   If $k=\max\{\M\hdim X,\M\hdim Y\}$, then,  again by Corollary \ref{1.1}, $\zz^k X$ and $\zz^k Y$ will be in $\M$.  Therefore $\M\hdim \zz^k Z \le 1$, and so $\zz^{k+1} Z$ will be in $\M$.   Thus by Corollary \ref{1.1}, $\M\hdim Z\le k+1$, proving (\ref{23}). 

Now take $k=\max\{\M\hdim Y,\M\hdim Z-1\}$.  Then $\zz^k Y$ and $\zz^{k+1} Z$ are in $\M$.  We take the pushout diagram
\[\xymatrix{
& 					& 									& 0 \ar[d]							& 0 \ar[d]						& \\
& 					& 									& \zz^{k+1} Z \ar@{=}[r] \ar[d]			& \zz^{k+1} Z \ar[d]				& \\
& 0 \ar[r]				& \zz^k X \ar[r] \ar@{=}[d]					& T \ar[r] \ar[d]						& F \ar[r] \ar[d]					& 0\\
& 0 \ar[r]				& \zz^k X \ar[r]							& \zz^k Y \ar[r] \ar[d]					& \zz^k Z  \ar[r] \ar[d]				& 0.\\
& 					& 									& 0									& 0								& \\
}\]
with $F$ free and hence in $\M$.  Since, by Corollary \ref{1.1},  $\zz^{k+1} Z$ and $\zz^k Y$ are in $\M$, so is $T$.  Since $F\in\M$, $\zz^k X$ has to also to be in $\M$.  Hence $\M\hdim X\le k$, and we have (\ref{21}).  
\end{proof}

For a subset $\M\sbe \md(R)$, let $\DD(\M)$ denote the category of modules $X$ such that $\M\hdim X$ is finite.  If $\M$ is resolving, then by Corollary \ref{1.1}, $\DD(\M)=\{X\in\md(R) \mid \zz^{\gg 0} X\in\M\}$.    The next result easily follows from the previous lemma.  
\begin{corollary}
Let $\M$ be resolving.  For any $n$, the set $\{X\in\md(R) \mid \M\hdim X\le n\}$ is resolving.  Furthermore, $\DD(\M)$ is thick, and $\thick(\M)=\DD(\M)$.
\end{corollary}




Through these results, we may construct many resolving and thick subcategories.  It is easy to show that the intersection of a collection of resolving subcategories and the intersection of a collection of thick subcategories are resolving and thick respectively.    The following lemma allows us to construct even more resolving subcategories.  For $\M\sbe\md(R)$, we say $\M_p=\{M_p\mid M\in\M\}$.
\begin{lemma}
Let $R$ and $S$ be rings and $F:\md(R)\to \md(S)$ be an exact functor with $F(R)=S$.  Then for any resolving subcategory $\M\sbe \md(S)$, $F^{-1}(\M)$ is a resolving subcategory of $\md(R)$.
\end{lemma}
The proof is elementary and is left to the reader.  Applying this lemma to the localization functor, for any $V\sbe\spec R$, the category of all $M\in \md(R)$ with $M_p$ free for all $p\in V$ is also resolving.  The following lemmas give insight into the behavior of  resolving categories under localization.  The first lemma is from \cite[Lemma 4.8]{Takahashi10} and \cite[Lemma 3.2(1)]{DaoTakahashi12}, and the second is from \cite[Proposition 3.3]{DaoTakahashi13}. 
\begin{lemma}\label{loc}
If $\M$ is a resolving subcategory, then so is $\add\M_\p$ for all $\p\in\spec R$.
\end{lemma}
\begin{lemma}\label{local}
The following is equivalent for a resolving subcategory $\M$ and a module $M\in\md(R)$.
\begin{enumerate}
\item $M\in \M$
\item $M_p\in\add\M_p$ for all $p\in\spec R$
\item $M_\mm\in\add\M_\mm$ for all maximal ideals $\mm$.  
\end{enumerate}
\end{lemma}
Recall the definition of $\Ld$ and $\Ga$ from the introduction.  These lemmas show that if $\M$ is resolving, then for all $f\in\Gamma$, $\Ld_\M(f)$ is a resolving subcategory.
\begin{corollary}
Set
\[\Ld_\M(f)=\{M\in\md(R)\mid \add \M_p\hdim M_p\le f(p)\quad  \forall p\in\spec R\}.\]
If $\M$ is resolving, then for all $f\in\Gamma$, $\Ld_\M(f)$ is a resolving subcategory.
\end{corollary}
Let $\MCM$ denote the category of maximal Cohen-Macaulay modules.  As noted earlier, when $R$ is Cohen-Macaulay, $\MCM$ is resolving.  Furthermore, letting $d=\dim R$, $\zz^d M$ is in $\MCM$ for every $M\in\md(R)$.  Hence, $\DD(\MCM)=\md(R)$.  The following shows that dimension with respect to $\MCM$ is very computable.
\begin{lemma}\label{dimension}
Suppose $\M\sbe \N$ are resolving subcategories.  Then $\M$ is thick in $\N$ if and only if for any module  $X\in\DD(\M)$, we have $\M\hdim X=\N\hdim X$.  Furthermore, if $R$ is Cohen-Macaulay,  $\M$ is a thick subcategory of $\MCM$ if and only if dimension with respect to $\M$ satisfies the Auslander Buchsbaum Formula, i.e. for all $X\in\DD(\M)$  we have 
\[\M\hdim X+\depth X=\depth R.\]
\end{lemma}
\begin{proof}
Suppose $M$ is thick in $N$ and  $X\in\DD(\M)$.  Then we may write  $0\to M_d\to\cdots\to M_0\to X\to 0$ with $M_i\in\M$ and $d=\M\hdim X$.  Since each $M_i$ is also in $\N$, we have $\N\hdim X\le d$.  Setting $e=\N\hdim X\le d$, by Corollary \ref{1.1}, there exists a $N\in\N$ such that 
\[0\to M_d\to \cdots\to M_e\to N\to 0\quad\quad\quad0\to N\to M_{e-1}\to\cdots\to M_0\to X\to 0\]
are exact.  However, since $\M$ is thick in $\N$, $N$ is also in $\M$, which implies that $e=d$, proving the only if part of the statement.

Now suppose that $\M\hdim X=\N\hdim X$ for all $X\in\DD(M)$.  Now suppose $0\to L\to M\to N\to 0$ is exact with $L,M\in\M$ and $N\in\N$.  Then $N\in\DD(\M)$ and $\M\hdim N=\N\hdim N=0$.  Therefore $N\in \M$, and so $\M$ is thick in $\N$.  

Assume $R$ is Cohen-Macaulay.  Let $\M$ be a resolving subcategory whose dimension satisfies the Auslander Buchsbaum formula.   Then for any module  $M\in\DD(\M)\cap \MCM$,  we have 
\[\depth R=\M\hdim M+\depth M=\M\hdim M+\depth R.\]
Thus $\M\hdim M=0$ forcing $M$ to be in $\M$.  Hence $\M$ is contained in $\MCM$.  

By what we have proved so far, it suffices to show that dimension with respect to $\MCM$ satisfies the Auslander  Buchsbaum formula.  But this follows from Corollary \ref{1.1}.
%
\end{proof}
Recall the definition of $\Phi$ and $\Ga$ from the introduction.  If dimension with respect to $\add \M_p$ satisfies the Auslander Buchsbaum formula for all $p\in\spec R$, then for all $\X\sbe \DD(\M)$, $\Phi_\M(\X)$ is in $\Gamma$.  
Before proceeding, we need one more definition and a result.  
\begin{definition}
Let $\A\sbe\M$.  We say $\A$ cogenerates $\M$, if for every $M\in\M$, there exists an exact sequence 
$0\to M\to A\to M'\to 0$
with $M'\in\M$ and $A\in\A$.
\end{definition}
The following is an important theorem from \cite[Theorem 1.1]{AuslanderBuchweitz89}.
\begin{theorem}\label{hull}
Suppose $\A$ and $\M$ are resolving with $\A\sbe\M$. If $\A$ cogenerates $\M$, then for every $X\in\DD(\M)$ with $\M\hdim X=n$, there exists an $A\in\DD(\A)$ and $M\in\M$ such that $\A\hdim A=n$ and $0\to X\to A\to M\to 0$ is exact. 
\end{theorem}

\section{Preliminaries: Semidualizing Modules}\label{beta}
We fix a module $C\in\md(R)$ and write $M^\da=\hm(M,C)$. 
\begin{definition}\label{tot}
A finitely generated module $X$ is totally $C$-reflexive if it satisfies the following.
\begin{enumerate}
\item\label{31} $\ext^{>0}(X,C)=0$
\item\label{32}  $\ext^{>0}(X^\da,C)=0$
\item\label{33} The natural homothety map $\eta_X:X\to X^{\da\da}$  defined by $\mu\mapsto (\ph\mapsto\ph(\mu))$ is an isomorphism.  
\end{enumerate}
Let  $\G_C$ denote the category of totally $C$-reflexive modules.
\end{definition}
The set $\G_C$ is essentially the subcategory over which $\da$ is a dualizing functor.  The notion of totally $C$-reflexivity  generalizes Gorenstein dimension zero.  In fact, when $C=R$, $\G_R$ is simply the category of Gorenstein dimension zero modules, which are also known as totally reflexive modules.  See \cite{Masek99} for further information on the subject.  The following proposition shows us that $\G_C$ is almost resolving.
\begin{lemma}
The set $\G_C$ is closed under direct sums, summands, and extensions.  
\end{lemma}
\begin{proof}
It is easy to show that $\G_C$ is closed under direct sums and direct summands.  Suppose we have
\[0\to X\to Y\to Z\to 0\]
with  $X,Z\in\G_C$. It is easy to check that $Y$ satisfies condition \ref{31} of Definition \ref{tot}.  We have 
\[0\to Z^\da\to Y^\da\to X^\da\to 0\quad\quad\quad0\to X^{\da\da}\to Y^{\da\da}\to Z^{\da\da}\to 0.\]
From the first exact sequence, it is easy to see that $Y$ satisfies condition \ref{32} of Definition \ref{tot}.  We can then use the five lemma to show that $Y$ satisfies condition \ref{33} of Definition \ref{tot}.
\end{proof}
In general, $\G_C$ will not be resolving. For example, if $C=R/xR$ for a regular element $x\in R$, $\ext^1(R/xR,R/xR)=R/xR\ne 0$. So $R$ cannot be in  $\G_{R/xR}$, and thus $\G_{R/xR}$ cannot be resolving. It is clear from the definition that $R\in\G_C$ is a necessary condition for $\G_C$ to be resolving.  In fact, this condition is sufficient.  
\begin{prop}
The set $\G_C$ will be resolving if and only if $\G_C$ contains $R$.  
\end{prop}
\begin{proof}
If $\G_C$ is resolving, by definition it contains $R$, so we prove the converse. So suppose $R$ is in $\G_C$.  In light of the last lemma, we need only to prove that if   $0\to X\to Y\to Z\to 0$ is exact with $Y,Z\in\G_C$, then $X$ is in $\G_C$ as well.  Since $Y$ and $Z$  satisfy condition (\ref{31}) of Definition \ref{tot}, it is easy to show that $X$ does too.  Also, since $\ext^1(Z,C)=0$, we have 
\[0\to Z^\da\to Y^\da\to X^\da\to 0.\]
Hence, we have the following commutative diagram with exact rows.
\[\xymatrix{
& 0	\ar[r]	& X	\ar[r] \ar[d]^{\eta_X}			& Y	\ar[r] \ar[d]^{\eta_Y}		& Z	\ar[r] \ar[d]^{\eta_Z}		& 0					& 		\\
& 0	\ar[r]	& X^{\da\da}	\ar[r]				& Y^{\da\da}	\ar[r]			& Z^{\da\da}	\ar[r]			& \ext^1(X^\da,C)	\ar[r] & 0\\
}\]
Since $\eta_Y$ and $\eta_Z$ are isomorphisms, the five lemma shows that $\eta_X$ is too, and that $\ext^1(X^\da,C)=0$.  Thus $X$ satisfies condition (\ref{33}) of Definition \ref{tot}.  It is easy to check using the first exact sequence that $\ext^{>1}(X^\da,C)=0$, showing that $X$ satisfies condition (\ref{22}) of Definition \ref{tot}.  
\end{proof}
Motivated by this proposition, we say that a module, $C$, is semidualizing if $R$ is in $\G_C$.  This is easily seen to be equivalent to the following definition which is standard in the literature.
\begin{definition}\label{sdm}
A module $C$ is semidualizing if $\ext^{>0}(C,C)=0$ and $R\cong \hm(C,C)$ via the map $r\mapsto (c\mapsto rc)$.
\end{definition}
For the remainder of the paper, we will let $C$ denote a semidualizing module.  Semidualizing modules were first discovered by Foxby in \cite{Foxby72} and were later rediscovered in different guises  by various authors, including Vasconcoles in  \cite{Vasconcelos74}, who called them spherical modules, and Golod, who called them suitable modules.  For an excellent treatment of the general theory of semidualizing modules, see \cite{Sather-Wagstaff09b}.  Examples of semidualizing modules include $R$ and dualizing modules.  If $R$ is Cohen-Macaulay and $D$ is a dualizing module, then $\G_D$ is simply $\MCM$.   Dimension with respect to $\G_C$ is often called Gorenstein $C$-dimension, or $\mbox{G}_C$-dimension for short, since it is a generalization of Gorenstein dimension.    We would expect $\mbox{G}_C$ and Gorenstein dimension to have similar properties.   Thus we have the following lemma, which is an easy exercise, and proposition, which is from \cite[Theorem 1.22]{Gerko01}. 
\begin{lemma}\label{cdim}
If $X\in\DD(\G_C)$, then $\G_C\hdim X=\min\{n\mid \ext^{>n}(X,C)=0\}$.  
\end{lemma}
\begin{prop}\label{ABform}
For any semidualizing module $C$,  $\mbox{G}_C$-dimension satisfies the Auslander Buchsbaum formula, i.e. for any module $X\in\DD(\G_C)$, we have 
\[\G_C\hdim X+\depth X=\depth R.\]
\end{prop}
In light of Lemma \ref{dimension}, when $R$ is Cohen-Macaulay this means that $\G_C$ is a thick subcategory of $\MCM$.  Interest in understanding $\mbox{G}_C$-dimension and the structure  of $\G_C$ is not new.  The following conjecture by Gerko from \cite[Conjecture 1.23]{Gerko01} is equivalent to saying that $\G_R$ is a thick subcategory of $\G_C$. 
\begin{conjecture}\label{lucho}
If $C$ is semidualizing, then for any module $X$, $\G_C\hdim X\le \G_R\hdim X$, and equality holds when both are finite.  
\end{conjecture}
We give one more construction in this section.  Take any $X\in\G_C$.  Then we have $0\to\zz X^\da\to R^n\to X^\da\to 0$ is exact.  Since $R^\da\cong C$, applying $\da$ yields the exact sequence
\[0\to X\to C^n\to (\zz X^\da)^\da\to 0.\]
Hence $\G_C$ is cogenerated by $\add C$.  Furthermore, if $F_\bullet$ is a projective resolution of $X^\da$ with $X\in\G_C$, then $F_\bullet^\da$ is an $\add C$ coresolution of $X$.  Splicing this together with a free resolution $G_\bullet$ of $X$, we get what is called a complete $PP_C$ or a complete $P_C$-resolution of $X$.  See \cite{White10} or \cite{Sather-Wagstaff09b} for more on the matter.

Before proceeding, we summarize the notations of this paper.
\begin{enumerate}
\item $R$ is a commutative noetherian ring
\item $\PP$ is the  subcategory of projective $R$-modules
\item $\Ga$ is the set of grade consistent functions
\item $\M\hdim X$ is the dimension of $X$ with respect to the category $\M\sbe\md(R)$
\item $\add \M$ is the smallest category closed under direct sums and summands containing $\M\sbe \md(R)$ 
\item $\Ld_\M(f)=\{X\in\md R\mid \add\M_p\hdim X_p\le f(p)\quad \forall p\in \spec R\}$ with $f\in \Ga$
\item $\Phi_\M(\X)(p)=\sup\{ \add\M_p\hdim X_p\mid X\in\X\}$ with $\M,\X\sbe \md(R)$ subcategories
\item $\DD(\M)=\{X\in\md R\mid \M\hdim X<\infty\}$ with $\M\sbe \md(R)$ a category
\item $\R(\M)=\{\X\sbe \md(R)\mid \M\sbe \X\sbe\DD(\M)\ \ \X\mbox{ is resolving}\}$
\item $\R$ the collection of resolving subcategories
\item $\thick_\N(\M)$ the smallest thick subcategory of $\N$ containing $\M$ with $\M\sbe\N\sbe \md(R)$ subcategories
\item $C$ is a semidualizing module 
\item $\G_C$ the collection of totally $C$-reflexive modules
\item $X^\da=\hm(X,C)$
\item For a resolving subcategory $\A$ and a module $M\in\md(R)$, set $\Ares M=\res(\{M\}\cap \A)$.
\end{enumerate}

\section{Comparing Resolving Subcategories}\label{gamma}
For the entirety of this section, let $\A$, $\M$, and $\N$ be resolving subcategories.  Recall that $\R(\A)$ is the collection of resolving subcategories $\X$ such that $\A\sbe \X\sbe\DD(\A)$. In this section, we compare $\R(\A)$ and $\R(\M)$ when $\A$ is contained in $\M$.  If $\A\sbe\M$, we may define $\eta^\M_\A:\R(\A)\to \R(\M)$ by $\X\mapsto \res(\X\cup\M)$ and $\rho_\A^\M:\R(\M)\to\R(\A)$ by $\X\mapsto\X\cap\DD(\A)$.  Note that if $\A\sbe\M\sbe\N$, then $\eta_\A^\N=\eta_\M^\N\eta_\A^\M$ and $\rho_\A^\N=\rho_\A^\M\rho_\M^\N$.  
\begin{prop}\label{rhoinjective}
If $\A$ cogenerates $\M$, then the map $\rho_\A^\M$ is injective.
\end{prop}
\begin{proof}
Suppose that for $\X,\Y\in\R(\M)$, we have $\rho^\M_\A(\X)=\rho^\M_\A(\Y)$, i.e. $\X\cap\DD(\A)=\Y\cap\DD(\A)$.  Take any $X\in \X$.  Since $X\in\DD(\M)$ and $\A$ cogenerates $\M$, by Theorem \ref{hull}, there exists $A\in\DD(\A)$ and $M\in\M$ such that $0\to X\to A\to M\to 0$ is exact.  Since $M\in\M\sbe \X$ and $X\in\X$, we know that $A$ is also in $\X$.  But then $A$ is in $\X\cap\DD(\A)=\Y\cap\DD(\A)$ and thus also in $\Y$.  Since $M\in\M\sbe\Y$, we know that $X$ must also be in $\Y$.  Hence $\X\sbe \Y$, and, by symmetry, we have equality.  Therefore, $\rho^\M_\A$ is injective. 
\end{proof}
In certain circumstances, this map is a bijection.  The following is Theorem A from the introduction.
\begin{theorem}\label{engine}
Let $\Psi$ be a set of increasing functions from $\spec R$ to $\NN$.  Suppose, $\A\sbe \M$ such that $\A$ cogenerates $\M$ and $\add \A_p$ is thick in $\add \M_\p$ for all $\p\in\spec R$.  If $\Phi_\A$ and $\Ld_\A$ are inverse functions giving a bijection between $\R(\A)$ and $\Psi$, then the following diagram  commutes.
\[\xymatrix{
\R(\M) \ar[dr]^{\Phi_\M} & &\\
& \Psi &\\
\R(\A) \ar[ur]^{\Phi_\A} \ar[uu]^{\eta_\A^\M}& &\\
}\]
Furthermore,  $\Ld_\M$ and $\rho_\A^\M$ are the respective inverses of $\Phi_\M$ and  $\eta_\A^\M$.
\end{theorem}
The proof of this result will be given after this brief lemma.
\begin{lemma}\label{jn}
If $\X$ and $\Y$ are subcategories and $\M$ is resolving, then $\Phi_\M(\res(\X\cup\Y))=\Phi_\M(\X)\jn\Phi_\M(\Y)$.
\end{lemma}
\begin{proof}
Since every element in $\res(\X\cup\Y)$ is obtained by taking extensions, syzygies, and direct summands a finite number of times, and since these operations never increase the $\M$ dimension, we have $\Phi_\M(\res(\X\cup\Y))\le\Phi_\M(\X)\jn\Phi_\M(\Y)$.  However, since $\X,\Y\sbe \res(\X\cup\Y)$, we actually have equality.
\end{proof}
\begin{proof}[Proof of Theorem \ref{engine}]
First, we will show that $\rho_\A^\M$ and $\eta_\A^\M$  are inverse functions and are thus both bijections.  Proposition \ref{rhoinjective} shows that $\rho_\A^\M$ is injective. Fix $\X\in \R(\A)$ and let $\Z=\rho_\A^\M\eta_\A^\M(\X)=\res(\X\cup\M)\cap\DD(\A)$.  It suffices to show that $\Z=\X$.    Setting $f=\Phi_\A(\X)$, this is equivalent to showing that $\Phi_\A(\Z)=f$, since $\Phi_\A$ and $\Ld_\A$ are inverse functions.  Since $\X\sbe\Z$, we know that $\Phi_\A(\Z)\ge f$.  From Lemma \ref{jn}, we have
\[\Phi_\M(\res(\X\cup\M))=\Phi_\M(\X)\jn\Phi_\M(\M)=\Phi_\M(\X).\]
Furthermore, since $\add\A_p$ is thick in $\add\M_\p$ for all $p\in\spec R$, Lemma \ref{dimension} implies that $\add \A_p\hdim A$ and $\add \M_p\hdim A$ are the same for all $p\in\spec R$ and $A\in \DD(\A)$.  Hence $\Phi_\A(\X)=\Phi_\M(\X)$ and $\Phi_\A(\Z)=\Phi_\M(\Z)$.  Therefore,
\[f\le \Phi_\A(\Z)=\Phi_\M(\Z)\le\Phi_\M(\res(\X\cup \M))=\Phi_\M(\X)=\Phi_\A(\X)=f\]
and so, $\Phi_\A(\Z)=f$.  Hence,  $\rho_\A^\M$ and $\eta_\A^\M$  are inverse functions.   Also, this argument shows that $\Phi_\A(\X)=\Phi_\M(\res(\X\cup \M))=\Phi_\M(\eta_\A^\M(\X))$, showing that the diagram commutes and that $\Phi_\M$ is also  a bijection.

It remains to show that $\Ld_\M={\Phi_\M}^{-1}$.  For any $f\in\Psi$, $\eta^\M_\A(\Ld_\A(f))$ is contained in $\Ld_\M(f)$.  Because $\Phi_\M$ is an increasing function and both $\Phi_\A$ and $\Ld_\A$ are inverse functions, we have
\[f=\Phi_\A\Ld_\A(f)=\Phi_\M(\eta^\M_\A(\Ld_\A(f))\le\Phi_\M\Ld_\M(f)\le f.\]
Thus we have $\Phi_\M\Ld_\M(f)= f$, and we are done.
\end{proof}
For a resolving subcategory $\A$, let $\Ss(\A)$ be the collection of resolving subcategories $\M$ such that  $\M$ and $\A$ satisfy the hypotheses of Theorem \ref{engine}, i.e. $\A$ cogenerates $\M$ and $\add \A_\p$ is thick in $\add \M_\p$ for all $p\in\spec R$.  The following theorem shows that we can patch together the bijections in Theorem \ref{engine}.
\begin{theorem}\label{main}
Let $\A$ be a resolving subcategory and  $\Psi$ be a set of increasing functions from $\spec R$ to $\NN$.  If $\Phi_\A$ and $\Ld_\A$ are inverse functions giving a bijection between $\R(\A)$ and $\Psi$,  then the following is a bijection
 \[\Ld:\Ss(\A)\x\Psi\to \bigcup_{\M\in\Ss(\A)} \R(\M)\sbe \R.\] 
 Furthermore, for any $\M,\N\in\Ss(\A)$ with $\M\sbe \N$, the following diagram commutes and  $\rho_\M^\N$ is the inverse of  $\eta_\M^\N$. 
\begin{equation}\label{Aa}
\xymatrix{
\R(\N) \ar[dr]^{\Phi_\N} 							& 		&\\
\R(\M) \ar[r]^{\Phi_\M} \ar[u]^{\eta_\M^\N}		& \Psi &\\
\R(\A) \ar[ur]_{\Phi_\A} \ar[u]^{\eta_\A^\M}		& 		&\\
}
\end{equation}
\end{theorem}
 Before we proceed with the proof of Theorem \ref{main}, we need a lemma.
\begin{lemma}
The set $\Ss(\A)$ is closed under intersections.
\end{lemma}
\begin{proof}
Let $\M,\N\in\Ss(\A)$.  Take any  $p\in\spec R$.  Suppose $0\to A_1\to A_2\to A_3\to 0$ is an exact sequence of $R_p$ modules with $A_1,A_2\in\add\A_p$ and $A_3\in\add (\M\cap\N)_p$.  Then $A_3$ is in $\add \M_p$.  Therefore, since $\add\A_\p$ is thick in $\add\M_p$ by assumption, $A_3$ is in $\add\A_\p$.  Since $\add\A_\p$ is resolving and contained in $\add (\M\cap\N)_p$,  $\add\A_\p$ is thick in $\add (\M\cap\N)_p$.

It remains to show that  $\A$ cogenerates $\M\cap\N$.  Take $X\in \M\cap\N$.  We have
\[0\to X\to A\to M\to 0\quad \quad\quad 0\to X\to A'\to N\to 0\]
with $M\in\M$, $N\in\N$, and $A,A'\in\A$.  Consider the following pushout diagram.
\[\xymatrix{
		& 0 \ar[d]		 			& 0 \ar[d]					& 						& \\
0\ar[r]	& X	\ar[r] \ar[d]				& A \ar[r] \ar[d]				& M  \ar[r] \ar@{=}[d]		& 0	\\
0\ar[r]	& A' \ar[r] \ar[d]				& T \ar[r] \ar[d]				& M \ar[r] 				& 0\\	
		& N \ar@{=}[r] \ar[d]			& N \ar[d]					& 						& \\
		& 0		 					& 0 							& 						& \\
}\]
It is easy to see $T\in\M\cap\N$.  We also have the exact sequence
\[0\to X\to A\op A'\to T\to 0\]
Since $A\op A'\in\A$, this completes the proof. 
\end{proof}
\begin{proof}[Proof of Theorem \ref{main}]
Suppose $\M,\N\in\Ss$ with $\M\sbe\N$.  From Theorem \ref{engine}, the following diagrams commute.
\[\xymatrix{
\R(\M) \ar[r]^{\Phi_\M}							& \Psi &\\
\R(\A) \ar[ur]_{\Phi_\A} \ar[u]^{\eta_\A^\M}		& 		&\\
}
\quad\quad\quad\quad\quad
\xymatrix{
\R(\N) \ar[r]^{\Phi_\N} 		& \Psi &\\
\R(\A) \ar[ur]_{\Phi_\A} \ar[u]^{\eta_\A^\N}		& 		&\\
}
\]
From here, it is easy to show that diagram (\ref{Aa}) commutes and $\Phi_\N$ and  $\eta_\M^\N$ are bijections with $(\eta_\M^\N)^{-1}=\rho_\M^\N$.

Also, Theorem \ref{engine} shows that  $\im(\Lambda)=\bigcup_{\M\in\Ss}\R(\M)$.  It remains to show that $\Lambda$ is injective.  Suppose $\X=\Ld_\M(f)=\Ld_\N(g)$.  Then we have $\M\sbe \X$ and $\N\sbe \X$, and hence $\M\cap\N\sbe\X$.  For any $X\in\X$ and any $n$ greater than $\M\hdim X$ and $\N\hdim X$, $\zz^n X$ is in $\M\cap\N$ by Corollary \ref{1.1}.  Therefore, $\X$ is contained in $\DD(\M\cap\N)$ and thus $\X\in\R(\M\cap\N)$.  The previous lemma tells us that $\M\cap\N$ is in $\Ss(\A)$, and so $\Ld_{\M\cap\N}:\Psi\to\R(\M\cap\N)$ is a bijection, by Theorem  \ref{engine}.  So there exists an $h\in\Psi$ such that $\Ld_{\M\cap\N}(h)=\Z=\Ld_\M(f)=\Ld_\N(g)$.  Therefore, we may assume that $\M$ is contained in $\N$.  

Since $\X\in\R(M)$ and $\X\in\R(\N)$,  we have $\N\sbe \X\sbe\DD(\M)$.   Thus, because $\A\sbe\M\sbe \N$, we have 
\[\N=\N\cap\DD(\M)=\rho_\M^\N(\N)=\eta_\A^\M\rho_\A^\M\rho_\M^\N(\N)=\eta_\A^\M\rho_\A^\N(\N)=\eta_\A^\M(\A)=\M\]
Since $\Ld_\M$ is injective, we then also have $f=g$.  
\end{proof}

As mentioned earlier, it is shown in \cite{DaoTakahashi13} that we have $\Ld_\PP$ is a bijection from $\Ga$ to $\R(\PP)$.  In Section \ref{zeta} and Section \ref{eta} we apply Theorem \ref{main} when  $\A=\PP$, and show  that $\Ss(\PP)$ contains the collection of thick subcategories of $\G_R$.  The following results gives an alternative way of viewing Theorem \ref{main}.
\begin{prop}\label{3.1}
In the situation of Theorem \ref{main}, if $\Psi=\Ga$ and $\PP$ is thick in $\M$, then the following diagram commutes.
\[\xymatrix{
\Ss(\A)\x \Ga \ar[dd]_{\id_{\Ss(\A)}\x\Ld_\PP} \ar[dr]^{\Ld}& &\\
&  \R & \\
\Ss(\A)\x\R(\PP) \ar[ur]^{\Xi} & & \\
}\]where $\Xi(\M,\X)=\res(\M\cup\X)$.  Furthermore, $\id_{\Ss(\A)}\x\Ld_\PP$ is bijective and $\Xi$ is injective.  
\end{prop}
\begin{proof}
Since $\Ld_\PP$ is bijective, $\id_{\Ss(\A)}\x\Ld_\PP$ is too.  It suffices to show that for any $(\M,f)\in\Ss(\A)\x \Ga$ we have $\Xi(\M,\Ld_\PP(f))=\Ld_\M(f)$.  Set $\Z=\Xi(\M,\Ld_\PP(f))$.  First note that $\Z$ is in $\R(\M)$.  Since  $\PP$ is thick in $\M$ and hence in $\M$, by Lemma \ref{jn}, we have 
 \[\Phi_\M(\Z)=\Phi_\M(\res(\M\cup\Ld_\PP(f)))=\Phi_\M(\M)\jn\Phi_\M(\Ld_\PP(f))=\Phi_\PP(\Ld(\PP)(f))=f\]
  and thus $\Ld_\M(f)=\Z$, proving the claim.
 \end{proof}

\section{A generalization of the Auslander transpose}\label{delta'}

Let $C$ be a semidualizing module, and set $-^\da=\hm(-,C)$.  For the entirety of this section,  $\A$ denotes  a thick subcategory of $\G_C$ that is closed under $\da$. Recalling Proposition \ref{ABform}, $\A\hdim$ satisfies the Auslander Buchsbaum formula.  We set $\Ares M=\res(\{M\}\cup\A)$.

The Auslander transpose has been an invaluable tool in both representation theory and commutative algebra.  In this section, we generalize the notion of the Auslander transpose using semidualizing modules and list some properties which we will use.  The Auslander transpose has previously been generalised in  \cite{Geng13} and \cite{Huang99}, but the construction here is different. 
\begin{definition}
An {\it $\A$-presentation} of $X$ is an exact sequence $A_1\xto{\ph} A_0\to X\to 0$ with $A_1,A_0\in\A$.    Set $\atr X=\coker \ph^\da$. 
\end{definition}

When $C=R$, we get the usual Auslander transpose of $X$ which will be denotes $\tr X$.  The "functor" $\atr$ is not well defined up to isomorphism or even stable isomorphism, motivating a new equivalence relation. Finding the correct equivalence relation is actually a subtle affair.  The equivalence relation must make $\atr X$  be well defined, but it must also detect resolving subcategories.  For modules $X$ and $Y$, we write $X\sim'Y$ and $Y\sim'X$ if there exists an $A\in\A$ such that  $0\to X\to Y\to A\to 0$ is exact.  Let $\A$-equivalence, denoted by $\sim$, be the transitive closure of the relation $\sim'$.   Since $\sim'$ is  symmetric and reflexive, $\sim$ is an equivalence relation.  Stable equivalence implies $\A$-equivalence, and when $\A=\PP$, they are the same.  We will see in a moment that $\atr X$ has the desired properties.
 
 \begin{remark}
 
 We would like to think of $\atr$ as a functor.  However, $\md(R)$ modulo $\A$-equivalence does not form a sensible category.  However, a very similar construction is functorial.  Let $P\xto{\ph} Q\to X\to 0$ be a projective presentation.  Set $\Trc X=\coker \ph^\da$.  A similar construction is given in \cite{Geng13} and \cite{Huang99}.  We will briefly show that $\Trc:\md(R)/\A\to \md(R)/\A$ is a functor.  We thank the referee for bringing the following construction to our attention.  
 
 We give some definitions first.
 \begin{enumerate}
\item $\X/\Y$ is the category whose objects are $\X$, and whose morphisms are
\[\hm_{\X/\Y}(X_1,X_2):=\frac{\hm_\X(X_1,Y_1)}{F_\Y(X_1,X_2)}\]
where $X_1,X_2\in\X$ and $F_\Y(X_1,X_2)$ is the subgroup of morphisms in $\X$ which factor through an object in $\Y$.  
\item $\mor \X$ is the category whose objects are morphisms $f:X_1\to X_2$.  A morphism $(g_1,g_2)$ between objects $f:X_1\to X_2$  and $f':X'_1\to X'_2$ in $\mor \X$ is a pair of morphisms $g_1:X\to X'$ and $g_2:Y\to Y'$ such that the following diagram commutes.
\[\xymatrix{
X_1 \ar[r]^{g_1} \ar[d]^f 	& X'_1 \ar[d]^{f'} \\
X_2 \ar[r]^{g_2} 			& X'_2\\
}\]
\item For $f,f'\in \mor \X$, a morphism $(g_1,g_2):f\to f'$ is homotopically trivial if there exists an $h: X_2\to X'_1$ such that $f'hf=g_2f=f'g_1$.  Let $H_{\X}(f,f')$ denote the subgroup of homotopically trivial maps.
\item $\hmor\X$ is the category whose objects are the same as $\mor\X$ but whose morphisms are 
\[\hm_{\hmor\X}(f,f')=\hm_{\mor\X}(f,f')/H_\X(f,f')\]
\end{enumerate}

   We now mimic the construction of the Auslander transpose given in  \cite[Chapter 3, Section 1]{Auslander71}.  Set $\C=\add C$.  The functor $\da$ restricts to a functor $\da:\PP\to \C$ which induces a contravariant functor $\da:\mor \PP\to \mor \C$ given by $f\mapsto f^\da$.  It is easy to check that the group homomorphism $\da:\hm_{\mor\PP}(f,f')\to \hm_{\mor\C}({f'}^\da,f^\da)$ given by $(g_1,g_2)\mapsto (g_2^\da,g_1^\da)$ maps the subgroup $H_{\PP}(f,f')$ to $H_{\C}({f'}^\da,f^\da)$.  Therefore, $\da$ induces a functor $\hmor \PP\to \hmor\C$.  Furthermore, it is easy to check that the functor $\coker:\hmor\C\to \mod R/\C$ given by $f\mapsto \coker f$ is a well defined functor.  The discussion in  \cite[Chapter 3, Section 1]{Auslander71} indicates that there is a functor $\rho:\md(R)/\PP\to \hmor\PP$ which sends a module to a projective presentation.  We summarize these discussions with  the following commutative diagram.  
\[\xymatrix{
							& \mor\PP \ar[d] \ar[r]^{\da}		& \mor\C \ar[d]			& \\
\mod(R)/\PP \ar[r]^{\rho}	& \hmor\PP \ar[r]^{\da}			& \hmor\C \ar[r]^{\coker} & \md(R)/\C\\
}\]
The composition of the bottom row is $\Trc$.    Since $\A$ is closed under $\da$ and is thick in $\MCM$, $\Trc$ fixes $\A$.  Thus, since $\PP,\C\sbe \A$, it follows that $\Trc$ induces a functor $\md(R)/\A\to \md(R)/\A$ as desired.  

This approach has two deficiencies.  First, we cannot compute $\Trc$ using $\A$-resolutions.  We will use $\A$-resolutions for example in Lemma \ref{5.1} (4) to show that $\atr\atr X\sim X$.  Second, if $X$ and $Y$ are isomorphic in $\md(R)/\A$, it is not clear if $\Ares X=\Ares Y$.  Because of these issues, $\Trc$ cannot take the place of $\atr$ in this work.

\end{remark}

We proceed to show that $\A$-equivalence is sufficient for our purposes.  

\begin{prop}
For a module $X$, $\atr X$ is unique up to $\A$-equivalence.
\end{prop}
\begin{proof}
Let $\pi$ be the projective presentation $P_1\to P_0\to X\to 0$, and $\rho$ the $\A$-presentation $A_1\to A_0\to X\to 0$.  Suppose there is an epimorphism $\pi\to \rho$.    Then there exists the following commutative diagram
\begin{equation}\label{resB}
\xymatrix{
&			&					& 0	\ar[d]				& 0	\ar[d]				& 							&	\\
& 0 \ar[r]		& B_2 \ar[r]			& B_1 \ar[r] \ar[d]			& B_0 \ar[r]	\ar[d]		& 0							& \\
& 			& 					& P _1\ar[r]	\ar[d]		& P_0 \ar[r]	\ar[d]		& X \ar[r]	\ar@{=}[d]			& 0\\
&			& 					& A_1 \ar[r]	\ar[d]		& A_0 \ar[r]	\ar[d]		& X \ar[r] 					& 0\\		
& 			& 					& 0						& 0						& 							& 		\\
}
\end{equation}
whose columns are exact and $B_0,B_1,B_2$ in $\A$.  Applying $\da$ to the diagram yields
\[\xymatrix{
& 			& 								& 0	 \ar[d]					& 0	 \ar[d]				& 0	 \ar[d]					& 	\\
& 0	\ar[r]	& X^\da \ar[r] \ar@{=}[d]			& A_0^\da \ar[r] \ar[d]			& A_1^\da  \ar[r] \ar[d]		& \atr^\rho X \ar[r] \ar[d]		& 0\\
& 0 \ar[r]		& X^\da \ar[r]						& P_0^\da \ar[r] \ar[d]			& P_1^\da \ar[r] \ar[d]		& \atr^\pi X \ar[r] \ar[d]			& 0\\
& 			& 0 \ar[r]							& B_0^\da \ar[r] \ar[d]			& B_1^\da \ar[r] \ar[d]		& B_2^\da \ar[r] \ar[d]			& 0\\		
& 			& 								& 0							& 0						& 0							& 	\\
}\]
Where $\atr^\rho X$ and $\atr^\pi X$  denote $\atr X$ computed using $\rho$ and $\pi$ respectively. Since the rows are exact, and the middle two columns are exact, the snake lemma shows the last column is exact.  Since $B_2^\da\in\A$, we see that  $\atr^\rho X\sim \atr^\pi X$.

Consider any two $\A$-presentations, $\rho$ and $\rho'$.  It is easy to construct projective presentations $\psi$ and $\psi'$ with epimorphisms $\psi\to \rho$ and  $\psi'\to \rho'$. The proof of \cite[Proposition 4]{Masek99} shows that there is a projective presentation of $\pi$ and epimorphisms $\pi\to \psi$ and $\pi\to \psi'$.  Using our work so far, we know that $\atr^\rho X\sim \atr^\pi X\sim\atr^{\rho'}$. 
\end{proof}

\begin{lemma}\label{5.1}
For any $X,Y\in\md(R)$ such that $X\sim Y$, the following are true.
\begin{enumerate}
	\item $\Ares X=\Ares Y$
	\item $\zz X\sim \zz Y$
	\item $\atr X\sim\atr Y$
	\item $\atr\atr X\sim X$
\end{enumerate}
\end{lemma}
\begin{proof}
It suffices to assume that $0\to X\to Y\to A\to 0$ with $A\in\A$.  Proving (1) is trivial.   For suitable choices of syzygies, we have   $0\to \zz X\to \zz Y \to \zz A\to 0$.  Since $\zz A$ is in $\A$,  and since syzygies are unique up to stable, and hence $\A$-equivalence, this proves (2).  

Now we show (3).  Consider the diagram with exact rows
\[\xymatrix{
& 				& P_1 \ar[d] \ar[r]					& P_0 \ar[d] \ar[r]			& Y \ar[d] \ar[r]		& 0\\
& 0 \ar[r]			& \zz A  \ar[r]						& Q \ar[r]				& A \ar[r]				& 0\\
}\]
with $Q,P_0,P_1$ projective and surjective vertical arrows.  Using the snake lemma, we can extend this to the diagram
\[\xymatrix{
&				& 0 \ar[d]						& 0 \ar[d]				& 0 \ar[d]			&\\
& 				& B_1 \ar[d] \ar[r]					& B_0 \ar[d] \ar[r]			& X	\ar[d] \ar[r]		& 0\\
& 				& P_1 \ar[d] \ar[r]					& P_0 \ar[d] \ar[r]			& Y	\ar[d] \ar[r]		& 0\\
& 0 \ar[r]			& \zz A \ar[d] \ar[r]					& Q \ar[d] \ar[r]			& A	\ar[d] \ar[r]		& 0\\	
&				& 0 								& 0 						& 0					&\\	
}\]
such that $B_1,B_0$ are in $\A$. Applying $\da$ to this diagram gives the following.
\[\xymatrix{	
& 			& 						& 0	\ar[d]					& 0	\ar[d]							& 							& 	\\
& 0 \ar[r]		& A^\da \ar[r]				& Q^\da \ar[d] \ar[r]			& (\zz A)^\da \ar[d] \ar[r]				& \ext^1( A,C) \ar[r] \ar[d]		& 0\\
& 0 \ar[r]		& Y^\da \ar[r]				& P_0^\da \ar[d] \ar[r]		& P_1^\da \ar[d] \ar[r]				& \atr Y \ar[r]	\ar[d]			& 0\\
& 0 \ar[r]		& X^\da \ar[r]				& B_0^\da \ar[d] \ar[r]		& B_1^\da \ar[d] \ar[r]				& \atr X \ar[r] \ar[d]			& 0.\\
&			& 				 		& 0						& 0								& 0 							& 	\\
}\]
Since $\ext^1(A,C)=0$, applying the snake lemma to the middle two columns yields $\atr X\cong\atr Y$.  This proves (3).

To see (4), consider the projective presentation $P_1\to P_0\to X\to 0$.  Then $P_0^\da\to P_1^\da\to \atr X\to 0$ is an $\A$ presentation.  Using this presentation, to compute $\atr\atr X$ gives the result.

\end{proof}

We close this section with an example of a property shared by $\atr$ and $\tr$. 

 \begin{lemma}\label{2.3}
 
Let $0\to X\to Y\to Z\to0$ be an exact sequence in $\md R$.  For suitable choices of $\atr$, we have the exact sequence
\[0\to Z^\da\to Y^\da\to X^\da\to \atr Z\to\atr Y\to \atr X\to 0.\]
Furthermore, if $\ext^i(X,C)=0$, then we have
\[0\to\atr\zz^i Z\to\atr\zz^i  Y\to\atr\zz^i X\to 0.\]

\end{lemma}

\begin{proof}

Let $\ta$ denote the map from $Y$ to $Z$.   We have the short exact sequence $0\to\zz^i X\to\zz^i Y\xto{\zz^i\ta}\zz^i Z\to 0$ for all $i\ge 0$.  We can construct the following short exact sequence of $\A$ presentations.
\[\xymatrix{
& 0	\ar[d]					& 0	\ar[d]				& 0	\ar[d]						& \\
& A^0_1 \ar[d] \ar[r]			& A^0_0 \ar[d] \ar[r]		& \zz^i X	 \ar[d] \ar[r]				& 0\\
& A^1_1 \ar[d] \ar[r]			& A^1_0 \ar[d] \ar[r]		& \zz^{i} Y \ar[d]^{\zz^i\ta} \ar[r]		& 0\\
& A^2_1 \ar[d] \ar[r]			& A^2_0 \ar[d] \ar[r]		& \zz^i Z	 \ar[d] \ar[r]				& 0\\		
& 0							& 0						& 0								& \\
}\]
Applying $\da$ and also the snake lemma yields 
\[0\to(\zz^i Z)^\da\xto{(\zz^i\ta)^\da}\zz^{i} Y\xto{\ld}\zz^i X^\da \xto{\ep}\atr\zz^i Z\xto{\eta} \atr\zz^{i} Y\to\atr\zz^i X \to0.\] 
Setting $i=0$ at this stage gives us the first claim.  The short exact sequence $0\to \zz^i X\to\zz^{i}  Y\to\zz^i Z\to0$ gives the following long exact sequence of $\ext$ modules.   
\[0\to(\zz^i Z)^\da\xto{(\zz^i\ta)^\da}\zz^{i}  Y\xto{\ld}\zz^i X^\da\xto{\de}\ext^1(\zz^i Z,C)\xto{\ext^1(\zz^i\ta,C)}\ext^1(\zz^{i} Y,C)\to\cdots\]
We also have
\[\cdots\to \ext^i(X,C)\to \ext^{i+1}(Z,C)\xto{\ext^{i+1}(\ta,C)}\ext^{i+1}(Y,C)\to\cdots.\]
Since $\ext^i(X,C)=0$ by assumption,  $\ext^{i+1}(\ta,C)$ and $\ext^1(\zz^i\ta,C)$ are injective, forcing $\de$ to be zero.  Thus $\ld$ is surjective.  Then the first long exact sequence shows that $\ep$ is zero, and so $\eta$ is injective, giving the desired result.

\end{proof}

\section{Resolving Subcategories which are Maximal Cohen Macaulay On The Punctured Spectrum}\label{delta}
We keep the same conventions used in the previous section, except we also assume that  $(R,\mm,k)$ is a Noetherian local ring.  Recall that since $\A$ is a thick subcategory of $\G_C$, according to Proposition \ref{ABform}, dimension with respect to $\A$ satisfies the Auslander Buchsbaum formula.  Set $\Ares M=\res(\{M\}\cup\A)$, $\DA_0=\{M\in\DA\mid M_p\in\add\A_\p\  \forall p\in\spec R\del\mm\}$, and $\DA_0^i=\{M\in\DA_0\mid \A\hdim M\le i\}$.  This section is devoted to proving the following.  
\begin{theorem}\label{filtration} 
If $(R,\mm,k)$ is a local ring with $\dim R=d$, the filtration 
\[\A=\DA_0^0\sbne \DA_0^1\sbne\cdots\sbne\DA_0^d=\DA_0\]
 is a complete list of the resolving subcategories of $\DA_0$ containing $\A$.
\end{theorem}

This theorem and its proof is a generalisation of \cite[Theorem 2.1]{DaoTakahashi13}.  We now use results from the previous section to make the building blocks of the proof of Theorem \ref{filtration}.  
\begin{lemma}\label{2.2}
For any module $X\in\md R$, for suitable choices of $\atr X$ and $\zz\atr\zz X$, we have
\[0\to\ext^1(X,C)\to\atr X\to \zz\atr\zz X\to 0.\]
\end{lemma}
\begin{proof}
With $F_0,F_1,F_2$ projective, consider the sequence
\[F_2\xto{f} F_1\xto{g} F_0\to X\to 0.\]
We have $\coker g^\da=\atr X$.  By the universal property of kernel and cokernel, we have the following commutative diagram.  
\[\xymatrix{
& 0 \ar[r]		& \im g^\da \ar[r] \ar[d]^{\io}		& F_1^\da \ar[r]	 \ar@{=}[d]	& \atr X \ar[r] \ar[d]^{\ep}		& 0\\
& 0 \ar[r]		& \ker f^\da \ar[r]					& F_1^\da \ar[r]				& \im f^\da \ar[r]				& 0\\
}\]
The snake lemma yields the exact sequence
\[0\to\ker\io\to 0\to\ker\ep\to\ext^1(X,C)\to0\to\coker\ep\to0.\]
Thus $\ep$ is surjective and $\ker \ep\cong\ext^1(X,C)$, giving the exact sequence $0\to \ext^1(X,C)\to\atr X\to \im f^\da\to 0$.  It remains to show that $\im f^\da\sim \zz\atr\zz X$.  

We have the short exact sequence $0\to \im f^\da	\to F_2^\da\to \atr \zz X\to 0$.  Consider the pushout diagram
\[\xymatrix{
		&		 					& 0 \ar[d]					& 0 \ar[d]				& \\
		& 							& \zz\atr \zz X \ar@{=}[r] \ar[d]	& \zz\atr\zz X  \ar[d]		&	\\
0\ar[r]	& \im f^\da \ar[r]	\ar@{=}[d]	& T \ar[r] \ar[d]				& G \ar[r] \ar[d] 			& 0\\	
0\ar[r]	& \im f^\da \ar[r]				& F_2^\da \ar[r] \ar[d]			& \atr\zz X \ar[r] \ar[d]		& 0\\
		&		 					& 0 							& 0						& \\
}\]
with $G$ projective.  We have $\im f^\da\sim T\sim \zz\atr\zz X$ as desired

\end{proof}
\begin{lemma}\label{2.4}
If $X\in\DA_0$, for all $0\le i<\depth C$, for suitable choices of $\atr$, the following is  exact.
\[0\to\atr\zz^{i+1}\atr\zz^{i+1}X\to\atr\zz^i\atr\zz^i X\to\atr\zz^i\ext^{i+1}(X,C)\to0\]
\end{lemma}
\begin{proof}
Using Lemma \ref{2.2}, we have
\[0\to\ext^{i+1}( X,C)\to\atr \zz^iX\to \zz\atr\zz^{i+1} X\to 0.\]
Since $X\in\DA_0$,  $\ext^{i+1}( X,C)_p=0$ for every nonmaximal prime $p$.  Thus $\ext^{i+1}( X,C)$ has finite length, and so $\ext^i(\ext^{i+1}(X,C),C)=0$ for all $0\le i< \depth  C$.  Thus, we can apply Lemma \ref{2.3}.
\end{proof}

\begin{lemma}\label{2.41}

Let $X\in \DA_0$ and $0<n\le \depth C$, we have
\begin{align*}
\Ares(X,\atr\ext^1(X,C),& \atr \zz\ext^2(X,C),\cdots,\atr\zz^{n-1}\ext^n(X,C))\\
&= \Ares(\atr\zz^n\atr\zz^n X,\atr\ext^1(X,C),\atr \zz\ext^2(X,C),\cdots,\atr\zz^{n-1}\ext^n(X,C))
 \end{align*}
 
\end{lemma}

\begin{proof}

The previous lemma tells us that 
\begin{align*}
\Ares((\atr \atr X,\atr\ext^1(X,C)) &=\Ares(\atr\zz\atr\zz X,\atr\ext^1(X,C))\\
\Ares(\atr\zz\atr\zz X,\atr \zz\ext^2(X,C)) &=\Ares(\atr\zz^2\atr\zz^2 X,\atr\zz\ext^2(X,C))\\
&\vdots\\
 \Ares(\atr\zz^{n-1}\atr\zz^{n-1} X,\atr\zz^{n-1}\ext^n(X,C))&=\Ares(\atr\zz^n\atr\zz^n X,\atr\zz^{n-1}\ext^n(X,C))
 \end{align*}
Since $\atr\atr X\sim X$, the result is now clear.

\end{proof}

\begin{lemma}\label{2.5}
Let $0\le n<\depth R$ and $L$ be a nonzero finite length module.  There exists an $\A$-resolution $(G_\bullet,\dell^{L,n})$ of $\atr\zz^n L$ such that $G_i=0$ for all $i>n+1$ and 
\[\ker\dell_{i}^{L,n}=\atr \zz^{n-i}L\]
 for all $1\le i\le n$.  In particular, $\atr\zz^i L\in\Ares(\atr\zz^n L)$ for all $0\le i\le n$, $\A\hdim(\atr\zz^n L)=n+1$, and  $\atr\zz^n L\in \DA_0^{n+1}$.
 \end{lemma}
\begin{proof}
Let $(F_\bullet,\dell)$ be a free resolution of $L$.  Then we have
\[F_{n+1}\to F_n\to \zz^n L\to 0\quad\quad0\to\zz^n L\to F_{n-1}\xto{\dell_{n-1}}\cdots\xto{\dell_2}F_1\xto{\dell_1} F_0\to L\to 0.\]
Because $L$ has finite length, and since $\depth C=\depth R$ by Proposition \ref{ABform}, we have $\ext^i(L,C)=0$ for all $0\le i\le n$, and so we have the exact sequence
\[0\to L^\da\to F_0^\da\xto{\dell_1^\da}F_1^\da\xto{\dell_2^\da}\cdots\xto{\dell_{n-1}^\da}F_{n-1}^\da\to(\zz^n L)^\da\to 0.\]
Note that $L^\da=0$ since $L$ has finite length.  Thus, splicing this exact sequence with $0\to(\zz^n L)^\da\to F_n^\da\xto{\dell_{n+1}^\da} F_{n+1}^\da\to\atr\zz^n L\to 0$, we create an $\A$-resolution of $\atr\zz^n L$.  So we set $G_i=F_{n+1-i}^\da$ for $0\le i\le n+1$ and $G_i=0$ for $i>n+1$.  Set $\dell^{L,n}_i=\dell_{n+2-i}^\da$ for $1\le i\le n+1$ and $\dell^{L,n}_i=0$ for all $i>n+1$.  Using our previous arguments for values less that $n$, we see that  $\ker\dell_{i}^{L,n}=\atr \zz^{n-i}L$ for  $0\le i\le n$.  Showing the first two claims.  

It is now apparent that $\A\hdim\atr \zz^{n} L\le n+1$.  If $\ker\dell_n^{L,n}=\atr L$ is in $\A$, then so is $L$ since $\atr\atr L\sim L$.  However, this is impossible since $L^\da=\ext^0(L,C)=0$.

\end{proof}
\begin{lemma}\label{2.6}
For all $0\le n<\depth R$ and all nonzero finite length modules $L$, $\Ares \atr\zz^n L=\Ares\atr\zz^n k$.
\end{lemma}
\begin{proof}
Let $\ld$ denote the length function.  If $L\ne 0$, then we can write $0\to L'\to L\to k\to 0$ with $\ld(L')<\ld(L)$.  Since by Proposition \ref{ABform} $n< \depth R=\depth C$, we have $\ext^n(L',C)=0$, and so from Lemma \ref{2.3}, we have 
\[0\to\atr\zz^n k\to\atr\zz^n L\to\atr\zz^n L'\to 0.\]
Thus, by induction we have $\Ares\atr\zz^n L\sbe \Ares\atr\zz^n k$.

Now we wish to show $\atr\zz^n k\in\Ares\atr\zz^n L$. We proceed by double induction first on $\ld(L)$ and then on $n$.  The case $L=k$ is trivial, so suppose $\ld(L)>1$.   Write $0\to L'\to L\to k\to 0$ again.  Since $L'$ has depth zero, we can use Lemma \ref{2.5} to get the resolution $(G_\bullet,\dell^{L'})$.  Thus we have the exact sequence 
\[0\to \ker\dell_1^{L',n}\to G_0\to \atr\zz^n L\to 0.\]
Taking the pullback diagram with our last exact sequence yields the following.
\[\xymatrix{
& 					& 									& 0 \ar[d]							& 0 \ar[d]						& \\
& 					& 									& \ker\dell_1^{L',n} \ar@{=}[r] \ar[d]		& \ker\dell_1^{L',n} \ar[d]			& \\
& 0 \ar[r]				& \atr\zz^n k \ar[r] \ar@{=}[d]			& T \ar[r] \ar[d]						& G_0  \ar[r] \ar[d]					& 0\\
& 0 \ar[r]				& \atr\zz^n k \ar[r]						& \atr\zz^n L \ar[r] \ar[d]				& \atr\zz^n L'  \ar[r] \ar[d]			& 0.\\
& 					& 									& 0									& 0								& \\
}\]
It is now easy to see that it suffices to  show that $\ker\dell_1^{L',n}$ is in $\Ares\atr\zz^n L$. When $n=0$, $(G_\bullet,\dell^{L',n})$ is the resolution 
\[0\to G_1\xto{\dell_1^{L',0}} G_0\to\atr L'\to 0,\]
and we are done since $\ker\dell_1^{L',0}=G_1\in \A\sbe\Ares\atr L$.  So  suppose $n>0$.    We have $\ker\dell_1^{L',n}=\atr\zz^{n-1} L'$, by Lemma \ref{2.5}. By induction, $\Ares\atr\zz^{n-1}  L$ and $\Ares\atr\zz^{n-1} L'$ are the same as $\Ares\atr\zz^{n-1}  k$.  So we have $\ker\dell_1^{L',n}\in\Ares\atr\zz^{n-1}  L\sbe\Ares\atr\zz^{n}  L$, where the inclusion follows from Lemma \ref{2.5}, and we are done. 
\end{proof}

These next proofs are similar to those in \cite{DaoTakahashi13} with the appropriate changes.  They are included here for the sake of completeness.
\begin{prop}\label{2.9}
For every $0< n\le \depth R$, we have $\DA_0^n=\Ares\atr\zz^{n-1} L$ for every nonzero finite length module $L$. 
\end{prop}
\begin{proof}
By Lemma \ref{2.6}, we may assume that $L=k$.  By Lemma \ref{2.5}, we know that $\A\hdim(\tr\zz^{n-1} k)=n$.  Since localization commutes with cokernels, duals and syzygies, we have $\tr\zz^n k$ is in $\DA_0$ and hence in $\DA_0^n$.  Suppose $X\in \DA_0^n$.  Then $\zz^n X\in \A$, and so $\atr \zz^n \atr \zz^n X\in \A$.  Furthermore, for each  $i\ge0$,   $\ext^{i+1}(X,C)$ has finite length.  Hence, Lemma \ref{2.6} implies that $\atr\zz^i \ext^{i+1}(X,C)$ is in     $\Ares \atr\zz^i k\sbe\Ares\atr\zz^{n-1} k$, where the inclusion follows from Lemma \ref{2.5}.  By Lemma \ref{2.41}, we therefore have
\[X\in \Ares(\atr\zz^n\atr\zz^n X,\atr\ext^1(X,C),\atr \zz\ext^2(X,C),\cdots,\atr\zz^{n-1}\ext^n(X,C))\sbe \Ares\atr\zz^{n-1} k\]
which concludes the proof.

\end{proof}

We now prove the main result of this section.

\begin{proof}[Proof of Theorem \ref{filtration}]

We clearly have the chain $\A=\DA_0^0\sbne \DA_0^1\sbne\cdots\sbne\DA_0^d=\DA_0$.  Take $X\in\DA_0^n\del\DA_0^{n-1}$ for $d\ge n\ge 1$.  We need to show that $\Ares X=\DA_0^n$, and we have $\Ares X\sbe \DA_0^n$. We proceed by induction.  When $n=0$, the statement is  trivial.  So assume that $n>0$ and $\Ares \zz X=\DA_0^{n-1}$.  Since $\ext^n(X,C)$ has finite length, it suffices to show $\atr\zz^{n-1} \ext^n(X,C)\in\Ares X$, by Proposition \ref{2.9}. 

Since $\zz^n X\in \A$, the short exact sequence $0\to \zz^n X\to P\to \zz^{n-1} X\to 0$, with $P$ projective, is an $\A$ presentation of $\zz^{n-1} X$.  Using this presentation to compute $\atr$, we see that $\atr \zz^{n-1} X\sim \ext^1(\zz^{n-1}X,C)\cong \ext^n(X,C)$.  Therefore $\atr\zz^{n-1}\atr\zz^{n-1}X\sim \atr \zz^{n-1}\ext^n(X,C)$ by Lemma \ref{5.1}.  Therefore, it suffices to show that  $\atr\zz^{n-1}\atr\zz^{n-1}X\in \Ares X$, again by Lemma \ref{5.1}.

Let $0<i\le n-1$.  Since $\ext^i(X,C)$ has finite length, Lemma \ref{2.5} implies 
\[\atr\zz^{i-1}\ext^i(X,C)\in \DA_0^i\sbe \DA_0^{n-1}=\Ares\zz X\sbe \Ares X .\]
 Therefore, Lemma \ref{2.41} implies that
\[\atr\zz^{n-1}\atr\zz^{n-1}X\in\Ares(X,\atr\ext^1(X,C),\atr\zz\ext^2(X,C),\cdots,\atr\zz^{n-2}\ext^{n-1}(X,C))=\Ares(X)\]
as claimed.

\end{proof}

The following corollary is immediate from Theorem \ref{filtration}
\begin{corollary}\label{filtration'}
If $X\in \DA_0^n\del\DA_0^{n-1}$, then $\Ares X=\DA_0^n$.  
\end{corollary}

\section{Resolving Subcategories and Semidualizing Modules}\label{epsilon}
In this section, we keep the same notations and conventions as the previous sections, except we will not assume that $R$ is local.  In this section, we classify the resolving subcategories of $\DA$ which contain $\A$.  Note that it is easy to check that $C_p$ is a semidualizing $R_p$-module for all $p\in\spec R$.  In Lemma \ref{6.2}, we will see that for all $p\in\spec R$, $\add\A_p$ is a thick subcategory of $\G_{C_p}$  closed under $\hm_{R_p}(-,C_p)$.  The following is a modified version of \cite[Lemma 4.6]{DaoTakahashi12}, which is a generalisation of \cite[Proposition 4.2]{Takahashi09}.  For a module $X$, let $\NA(X)=\{p\in\spec R\mid X_p\notin\add\A_p\}$.
\begin{prop}\label{replace}
Suppose $X\in \DA$.  For every $p\in\NA(X)$, there is a $Y\in \Ares X$ such that $\NA(Y)=\V(p)$ and $\add \A_\pi\hdim Y_\pi=\add \A_\pi\hdim X_\pi$ for all $\pi\in V(p)$.
\end{prop}
\begin{proof}
If $\NA(X)=\V(p)$ we are done.  So fix a $q\in \NA(X)\del\V(p)$.   As in the proof of \cite[Lemma 4.6]{DaoTakahashi12}, choose an $x\in p\del q$ and consider the following pushout  diagram.
\[\xymatrix{
& 0 \ar[r]				& \zz X \ar[d]^{x} \ar[r]				& F \ar[d] \ar[r]	 	& X \ar@{=}[d] \ar[r]	& 0\\
& 0 \ar[r]				& \zz X \ar[r]						& Y \ar[r]				& X \ar[r] \ar[r]		& 0\\
}\]
with $F$ projective.  Immediately, we have $Y\in\Ares X$. Furthermore, $Y_{q'}\in \res X_{q'}$ for all $q'\in\spec R$.  Therefore, we have  $\NA(Y)\sbe \NA(X)$.  The proof of \cite[Lemma 4.6] {DaoTakahashi12} tells us that 
\[\depth(Y_\pi)=\min\{\depth(X_\pi),\depth(R_\pi)\}\]
for all $\pi\in V(p)$.  Thus, by Proposition \ref{ABform}, $\add \A_\pi\hdim Y_\pi=\add \A_\pi\hdim X_\pi$, for all $\pi\in V(p)$.
%
In particular, this shows that $V(p)$ is contained in $\NA(Y)$.

Localizing at $q$, yields the following.
\[\xymatrix{
& 0 \ar[r]				& \zz X_q \ar[d]^{x} \ar[r]				& F_q \ar[d] \ar[r]	 	& X_q \ar@{=}[d] \ar[r]	& 0\\
& 0 \ar[r]				& \zz X_q \ar[r]						& Y_q \ar[r]			& X_q \ar[r] \ar[r]		& 0\\
}\]
Note $x$ is a unit in $R_q$.  Thus, by the five lemma, $Y_q$ is isomorphic to $F_p$ and therefore is projective.  So we have $q\notin\NA(Y)$ and hence $\NA(Y)\sbne\NA(X)$.  

If $\NA(Y)\ne\V(p)$, then we may repeat this process and construct a $Y'$ that, like $Y$, satisfies all the desired properties except $\V(p)\sbe\NA(Y')\sbne\NA(Y)\sbne\NA(X)$.  Since  $\spec R$ is Noetherian, this process must stabilize after some iteration, producing the desired module.
\end{proof}
\begin{lemma}\label{sequence}
Let $V$ be a nonempty finite subset of $\spec R$.  Let $M$ be a module  and $\X$ a resolving subcategory such that $M_p\in\add\X_p$ for some $p\in\spec R$.  Then there exists exact sequences 
\[0\to K\to X\to M\to 0\quad\quad\quad 0\to L\to M\op K\op R^t\to X\to 0\]
with $X\in\X$ and $\NA(L)\sbe\NA(M)$ and $\NA(L)\cap V=\emptyset$.  
\end{lemma}
\begin{proof}
The result is essentially contained in the proof of \cite[Proposition 4.7]{Takahashi10}.  It shows the existence of the exact sequences and shows that $V$ is contained in the free locus of $L$ and thus $\NA(L)\cap V=\emptyset$.  Furthermore, the last exact sequence in the proof shows that for any $p\in \spec R$, $L_p$ is in $\res M_p$.  Hence, if $L_p$ is not in a resolving subcategory, then $M_p$ cannot be in that category as well, giving us $\NA(L)\sbe\NA(M)$.
\end{proof}

These lemmas help to prove the following proposition which is a key component of the proof of Theorem \ref{main2}.  This next result is also where we use Corollary \ref{filtration'} of the last section.  
\begin{prop}
Consider a module $M\in\md(R)$ and a resolving subcategory $\X\in\R(\A)$.  If for every $p\in\spec R$, there exists an $X\in\X$ such that $\add \A_\p\hdim M_p \le \add \A_\p\hdim X_p$, then $M$ is in $\X$.  
\end{prop}  
\begin{proof}
Because of Lemma \ref{local}, we may assume $(R,\mm,k)$ is local.  We proceed by induction on $\dim\NA(M)$.  If $\dim\NA(M)=-\infty$, then $M$ is in $\A$ and we are done.  Suppose $\dim\NA(M)=0$.  Then $M$ is in $\DA^t_0$ where $t=\A\hdim X$.  By Proposition \ref{replace}, there exists a $Y\in \Ares X\sbe\X$ with $\A\hdim Y=t$ and $Y\in\DA_0$, and thus $Y\in\DA_0^t\del\DA_0^{t-1}$.  By Corollary \ref{filtration'}, $\Ares{Y}=\DA_0^t$, and thus $M\in\Ares(Y)\sbe\X$.  

The rest of the proof uses Lemma \ref{sequence} and is identical to \cite[Theorem 3.5]{DaoTakahashi13}, except one replaces the nonfree locus of $M$ by $\NA(M)$ and replaces projective dimension by $\A\hdim$.  
\end{proof}
We come to the main theorem of this section.  Recall that $\Ga$ is the set of grade consistent functions.  
\begin{theorem}\label{main2}
Assume $R$ is Noetherian.  If $\A$ is a thick subcategory of $\G_C$ which is closed under $\da$, then $\Ld_\A$ and $\Phi_\A$ are inverse functions giving a bijection between $\Ga$ and $\R(\A)$.
\end{theorem}
\begin{proof}
The previous proposition shows that $\Ld_\A\Phi_\A$ is the identity on $\R(\A)$.  Let  $f\in \Ga$ and $p\in\spec R$.  Since $\add \A_p\hdim X_p\le f(p)$ for every $X\in\Ld_\A(f)$, we have $\Phi_\A(\Ld_\A(f))(p)\le f(p)$.  However, by \cite[Lemma 5.1]{DaoTakahashi13}  there is an $M\in\DD(\PP)\sbe\DD(\A)$ such that $\pd_{R_p} M_p=f(p)$ and $\pd_{R_q} M_q\le f(q)$ for all $q\in\spec R$.  Since for all $q\in\spec R$ $\pd_q M_q=\add\A_q\hdim M_q$, $M$ is in $\Ld_\A(f)$, and we have  $\Phi_\A(\Ld_\A(f))(p)= f(p)$.  Thus $\Phi_\A\Ld_\A$ is the identity on $\Ga$.  
\end{proof}

\section{Resolving Subcategories That Are closed under $\da$}\label{zeta}
We wish to expand upon Theorem \ref{main2} using the results in Section \ref{gamma}.   However, to use Theorem \ref{main2}, we need to understand which thick subcategories of $\G_C$ containing $C$ are closed under duals.    In this section, $C$ will be a semidualizing module.  Since $\G_C$ is cogenerated by $\add C$, as seen at the end of Section \ref{beta}, it stands to reason that the results of Section \ref{gamma} are applicable. 
\begin{lemma}\label{6.1}

Suppose  $\M\sbe\G_C$ is resolving with $C\in\M$.  Then $\M$ is thick in $\G_C$ if and only if for every $M\in\M$, $(\zz M^\da)^\da$ is in $\M$.  In particular, $\M$ is thick in $\G_C$ if any only if it is cogenerated by $\add C$.  
\end{lemma}
Since syzygies are unique up to projective summands, $(\zz M^\da)^\da$ is unique up to $\add C$ summands.  Thus, for our purposes, our choice of syzygy is inconsequential.  When $R=C$, $(\zz M^\da)^\da$ is the classical cosyzygy of a Gorenstein dimension zero module.  Thus in this case, the lemma is equivalent to saying that a resolving subcategory $\M$  of $\G_R$  is thick if and only if it is closed under cosyzygies.  
\begin{proof}
Assume $\M$ is thick, and let $M\in\M$. We have the following exact sequence.  
\[0\to \zz M^\da\to R^n\to M^\da\to 0\]
Applying $\da$ yields
\[0\to M\to C^n\to (\zz M^\da)^\da\to 0.\]
Since $C\in\M$, if $\M$ is thick in $\G_C$, $(\zz M^\da)^\da$ is in $\M$.  

Conversely, suppose for every $M\in\M$, $(\zz M^\da)^\da$ is in $\M$.  We wish to show that $\M$ is thick in $\G_C$.  Since $\M$ is resolving, it suffices to check that $\M^\da$ is also resolving, since $\da$ is a duality on $\G_C$.  It is also clear that $\M^\da$ is extension closed.  Since $C\in \M$, we have $R\in\M^\da$.  Therefore it suffices to check that $\M^\da$ is closed under syzygies.  Take $Z=M^\da\in\M^\da$.  Then since $(\zz M^\da)^\da$ is in $\M$,  $(\zz M^\da)^{\da\da}\cong \zz M^\da=\zz Z$ is in $\M^\da$, as desired.  

 \end{proof}
 
The following corollary, although intuitive, is not obvious, and it is not clear if it holds for other subcategories besides $\G_C$.

\begin{corollary}\label{6.2}

If $\M$ is thick in $\G_C$, then $\add \M_p$ is thick in $\G_{C_p}$ for all $p\in\spec R$.

\end{corollary}

\begin{proof}

Take $p\in\spec R$.  From Lemma $\ref{loc}$, we know that $\add\M_p$ is resolving.  By the previous lemma, it suffices to show that for all $M\in\add\M_p$, $(\zz_{R_p} M^\da)^\da=\hm(\zz_{R_p} \hm(M,C_p),C_p)$ is in $\add\M_p$.  For every $M\in\add\M_p$, there exists a $N$ such that $M\op N=L_p$ for some $L\in \M$.  Consider the following.
\begin{align*}
{(\zz L^\da)^\da}_p=& \hm(\zz_R \hm(L,C),C)_p= \hm(\zz_{R_p} \hm(L_p,C_p),C_p)\\
&\quad\quad\quad=\hm(\zz_{R_p} \hm(M\op N,C_p),C_p)=\hm(\zz_{R_p} \hm(M,C_p),C_p)\op\hm(\zz_{R_p} \hm(N,C_p),C_p)
\end{align*}
By the previous lemma, $(\zz L^\da)^\da$ is in $\M$, and so $(\zz_{R_p} M^\da)^\da$ is in $\add\M_p$.  

\end{proof}

\begin{prop}\label{6.3}

Let $\A$ be the smallest thick subcategory of $\G_C$ containing $C$.  Then $\A$ is closed under $\da$.  

\end{prop}

Since the intersection of thick subcategories of $\G_C$ is thick, it is clear that $\A$ exists.  

\begin{proof}

First, let $\W$ be the set of of modules obtained by applying $\da$ and $\zz$ to $R$ successive times.  Suppose for a moment that $\res \W=\A$.  Let $A\in \A$.  We will show that $A^\da\in \A$ by inducting on the number of steps needed to construct $A$ from $\W$.  See \cite{Takahashi09} for a precise definition of the notion of steps with regards to a resolving subcategory.  If $A$ takes 0 steps to construct, then $A$ is either $R$ or in $\W$, and the claim is clear. Suppose $A$ is constructed in $n>0$ steps.  Then there exists $B_1$ and $B_0$ which can be constructed in $n-1$ steps and satisfy one of the following situations.
\begin{enumerate}
\item $0\to A\to B_0\to B_1\to 0$ 
\item $0\to B_0\to A\to B_1\to 0$
\item $B_0=A\op B_1$
\end{enumerate}
Therefore one of the following is true.
\begin{enumerate}
\item[(a)] $0\to B_1^\da\to B_0^\da\to A^\da\to 0$
\item[(b)] $0\to B_1^\da\to A^\da\to B_0^\da\to 0$
\item[(c)] $B_0^\da=A^\da\op B_1^\da$
\end{enumerate}
By induction, $B_0^\da$ and $B_1^\da$ are in $\A$.  Since $\A$ is thick, each of these situatations implies that $A^\da$ is in $\A$.

Therefore, it suffices to show that $\res \W=\A$.  First, we show that $\res \W$ is a thick subcategory containing $C$.  In light of Lemma \ref{6.1}, it suffices to show that for every $A\in \res \W$, we have $(\zz A^\da)^\da\in\res\W$.  We work as we did in the previous paragraph, and we proceed by induction on the number of steps needed to construct $A$ from $\W$.  When it takes 0 steps, then $A$ is either $R$ or in $\W$, in which case the claim is clear.  Suppose $A$ needs $n>0$ steps to be constructed.  Working as we did in the previous paragraph, there exists modules $B_1$ and $B_0$ which can be constructed in $n-1$ steps and  satisfy one of (1), (2) or (3) above. Therefore, one of the following is true.
\begin{enumerate}
\item[(a)] $0\to (\zz A^\da)^\da\to (\zz B_0^\da)^\da\to(\zz B_1^\da)^\da\to 0$
\item[(b)] $0\to (\zz B_0^\da)^\da\to (\zz A^\da)^\da\to(\zz B_1^\da)^\da\to 0$
\item[(c)] $(\zz B_1^\da)^\da=(\zz A^\da)^\da\op(\zz B_0^\da)^\da$
\end{enumerate}
By induction $(\zz B_0^\da)^\da$ and $(\zz B_1^\da)^\da$ are in $\res\W$.  Since $\W$ is resolving, then so is $(\zz A^\da)^\da$ as desired.  

It suffices now to show that $\res\W\sbe \A$.  To do this, we show that each $W\in \W$ is  in $\A$.   We induct on $c(W)$, the smallest number of times it takes to apply $\zz$ and $\da$ to $R$ to obtain $W$.  If $c(W)=0$, then $W=R$, and we are done.  If $c(W)=1$, then $W$ is either $0$ or $C$ which are both in $\A$.  Therefore, we may assume that $c(W)>1$.  Then one of the following situations must occur.
\begin{enumerate}
\item $A=\zz^2 B$
\item $A=B^{\da\da}$
\item $A=\zz (B^\da)$
\item $A=(\zz B)^\da$
\end{enumerate}
where $c(B)=c(A)-2$.  By induction, $B$ is in $\A$.  In cases (1) and (2), it is clear that $A$ is in $\A$ too.  We have $c(B^\da)\le c(B)+1<c(A)$, and so $B^\da$ is in $\A$ by induction.  Now in case (3), the result  is clear.  So we assume that we are in case (4).  By Lemma \ref{6.1}, $(\zz( B^{\da\da}))^\da\cong (\zz B)^\da=A$ must be in $\A$.

\end{proof}

For the rest of this section, $\A$ will continue to be the smallest thick subcategory of $\G_C$ containing $C$.  It is immediate that $\A$ satisfies the assumptions of Theorem \ref{main2}.  We wish to apply the results from the  beginning of the paper.  Using the notation of Section \ref{gamma}, set $\Ss(C)=\Ss(\A)$, i$.$e$.$ let $\Ss(C)$ be the collection resolving subcategories $\M\sbe\md(R)$ such that $\A$ cogenerates $\M$ and $\add \A_p$ is thick in $\add \M_p$ for every $p\in \spec R$.

\begin{lemma}\label{6.4}

Every thick subcategory of $\G_C$ which contains $C$ is in $\Ss(C)$.  Furthermore, when $R$ is Cohen-Macaulay, every element in $\Ss(C)$ is contained in $\MCM$.  In particular,  when $C=D$ is a dualizing module, $\Ss(D)$ is the collection of thick subcategories of $\MCM$ containing $D$.  

\end{lemma}

\begin{proof}

Let $\M$ be a thick subcategory $\G_C$ containing $C$.  It is clear from the definition of $\A$ that $\M$ contains $\A$. By Lemma \ref{6.1}, $\M$ is cogenerated by $\add C$ and thus also by $\A$. By Corollary \ref{6.2}, $\add \M_p$ and $\add \A_p$ are thick in $\G_{C_p}$ for all primes $p\in\spec R$.  Therefore, dimension with respect to each of these subcategories satisfies the Auslander Buchsbaum formula.  It follows from Lemma \ref{dimension} that $\add \A_p$  is thick in $\add \M_p$ 

Now suppose that $R$ is Cohen-Macaulay and  $\X\in\Ss(C)$.  Since $\A$ cogenerates $\X$, for any $X\in\X$ there exists $0\to X\to A_0\to \cdots\to A_d\to X'\to 0$ with each $A_i\in\A$ and $d=\depth R$.  Since $\A\sbe \MCM$, $X$ is in $\MCM$.  The last statement is now clear, since in that case $\G_D=\MCM$.  

\end{proof}

We now come to the main results of the paper.

\begin{theorem}\label{main3}

Let $\A$ denote the smallest thick subcategory of $\G_C$ containing $C$.  For any  $\M\in\Ss(C)$ (e$.$g$.$ $\M$ is a thick subcategory of $\G_C$ containing $C$), $\Ld_\M$ and $\Phi_\M$ give a bijection between $\R(\M)$ and $\Ga$. 

Furthermore, the following is a bijection.
 \[\Ld:\Ss(C)\x\Ga\to \bigcup_{\M\in\Ss(C)} \R(\M)\sbe \R\] 
For any $\M,\N\in\Ss(C)$ with $\M\sbe \N$, then the following diagram commutes.
\[\xymatrix{
\R(\N) \ar[dr]^{\Phi_\N} 							& 		&\\
\R(\M) \ar[r]^{\Phi_\M} \ar[u]^{\eta_\M^\N}		& \Ga &\\
\R(\A) \ar[ur]_{\Phi_\A} \ar[u]^{\eta_\A^\M}		& 		&\\
}\]
In particular,  $\rho_\M^\N$ and  $\eta_\M^\N$ are inverse functions. 
\end{theorem}

\begin{proof}

Proposition \ref{6.3} states that $\A$ is a thick subcategory of $\G_C$ which contains $C$ and is closed under $\da$.  Therefore, by Theorem \ref{main2}, $\Ld_\A$ and $\Phi_\A$ give a bijection between $\R(\A)$ and $\Ga$.  The first statement is an application of Theorem \ref{engine} and Lemma \ref{6.4}.  The rest follows from Theorem \ref{main}.


\end{proof}
A resolving subcategory $\X$ is dominant if for every $p\in\spec R$, there is an $n\in\NN$ such that $\zz_{R_p}^n R_p/pR_p\in \add \X_p$.     
\begin{corollary}
Suppose $R$ is Cohen-Macaulay and has a dualizing module.  Then there is a bijection between resolving subcategories containing $\MCM$ and grade consistent functions.  Furthermore, the  following are equivalent for a resolving subcategory $\X$. 
\begin{enumerate}
\item\label{51} $\X$ is dominant 
\item\label{52} $\MCM\sbe \X$
\item\label{53} $\DD(\X)=\md(R)$
\end{enumerate}
\end{corollary}  
\begin{proof}
Letting $D$ be the dualizing module of $R$, $\MCM$ is the same as $\G_D$.  Hence, by the previous theorem, $\Ld_{\MCM}:\Ga\to\R(\MCM)$ is a bijection, showing the first statement. From \cite[Theorem 1.3]{DaoTakahashi13}, the following is a bijection. 
\[\xi:\Gamma\to\{\mbox{Dominant Resolving subcategories of } \md(R)\}\] 
\[\xi(f)=\{X\in\md(R)\mid \depth X\p \ge \h\p-f(\p)\}\]
It is clear that  $\xi(0)=\MCM$, hence every dominant subcategory contains \MCM.  Furthermore, we have $\mod(R)= \DD(\MCM)$, and hence every dominant resolving subcategory is an element of $\R(\MCM)$.   Then for any $f\in\Gamma$,  we have
\[\xi(f)=\{X\in\md(R)\mid \depth X_\p\ge \h\p-f(\p)\}=\{X\in\md(R)\mid \add\MCM_\p\hdim X_\p\le f(\p)\}=\Ld_{\MCM}(f).\]
Thus $\xi$ equals $\Ld_{\MCM}$, showing the equivalence of (\ref{51}) and (\ref{52}).

It is clear that (\ref{52}) implies (\ref{53}).  Assume (\ref{53}).  Take a $p\in\spec R$.  Then we have $\X\hdim R/p<\infty$.  This implies that $\zz^n R/p\in \X$ for some $n$.  Hence $\zz^n_{R_p}R_p/pR_p\in\add\X_p$, and so $\X$ is dominant.  
\end{proof}

\section{Gorenstein Rings and Vanishing of Ext}\label{eta}
In this section, $(R,\mm,k)$ is a local Gorenstein ring.  In this case, $\MCM$ is the same as $\G_R$, and  Lemma \ref{6.4} implies that $\Ss(R)$ is merely the collection of thick subcategories of $\MCM$.  This gives us the following which recovers  \cite[Theorem 7.4]{DaoTakahashi13}.
\begin{theorem}\label{maingor}
If $R$ is Gorenstein, then we have the following commutative diagram of bijections
\[\xymatrix{
\{\mbox{Thick subcategories of }\MCM\}\x \Ga \ar[dr]^{\Ld} \ar[dd]^{\Ld_\PP}& &\\
& \{\Z\in \R\mid \Z\cap \MCM \mbox{ is thick in } \MCM\} &\\
\{\mbox{Thick subcategories of }\MCM\}\x \R(\PP) \ar[ur]^{\Xi} & &\\
}\]
where $\Xi(\M,\X)=\res(\M\cup\X)$.
\end{theorem} 
\begin{proof}
Let $\mathfrak{Z}$ be the collection of resolving subcategories whose intersection with $\MCM$ is thick in $\MCM$.  As observed before the Theorem, $\Ss(R)$ is simply the thick subcategories of $\MCM$.  Since for any $\M\in\Ss(R)$, $\DD(\M)\cap \MCM$ is $\M$, the image of $\Ld$ lies in $\mathfrak{Z}$.  Furthermore, for any $\Z\in\mathfrak{Z}$, $\Z$ is in $\R(\Z\cap\MCM)$, thus the result follows from Proposition \ref{3.1} and Theorem \ref{main3}. 
\end{proof}
It is natural to ask when the image $\Ld$ is all of $\R$.  This will happen precisely when  every resolving subcategory of $\MCM$ is thick. This occurs, by \cite[Theorem 6.4]{DaoTakahashi13}, when $R$ is a complete intersection.  We will give a  necessary condition for $\im\Ld=\R$ by examining the resolving subcategories of the form
\[\M_\B=\{M\in\md(R)\mid \ext^{>0}(M,B)=0\quad \forall B\in\B\}\]
 where $\B\sbe\md(R)$. Dimension with respect to this category can be calculated in the following manner.
 \begin{lemma}\label{7.1}
 For all $\B\sbe \md(R)$, we have the following.  
 \[  \M_\B\hdim M=\inf\{n\mid \ext^{>n}(M,B)=0\quad\forall B\in\B\}\]
\end{lemma}
\begin{proof}
Let $M\in\md(R)$. For all $i>0$ and $j\ge 0$ and each $B\in\B$, we have $\ext^{i+j}(M,B)=\ext^i(\zz^j M,B)$.  So $\ext^{i+n}(M,B)=0$ for all $i\ge 0$ if and only if $\zz^n M$ is in $\M_\B$. 
\end{proof}
\begin{lemma}\label{motiv}
For any $\B\sbe\md(R)$, we have $\M_\B\cap \DD(\PP)=\PP$.  
\end{lemma}
\begin{proof}
To prove this, it suffices to show that if $\pd(X)=n>0$, then $\ext^n(X,B)\ne 0$.  Take a minimal free resolution
\[0\to F_n\xto{d} F_{n-1}\to\cdots\to F_0\to X\to 0.\]
Note that $\im(d)\sbe\mm F_{n-1}$.   We then get the complex
\[0\to\hm(X,B)\to\hm(F_0,B)\to\cdots\to\hm(F_{n-1},B)\xto{d^*}\hm(F_n,B)\to 0.\]
Now $\im(d^*)$ will still lie in $\mm\hm(F_n,B)$, and thus by Nakayama, $d^*$ cannot be surjective.  Hence we have $\ext^n(X,B)=\coker d^*\ne 0$.  
\end{proof}
Araya in  \cite{Tokuji12}  defined AB dimension by $\ABdim M=\max\{b_M,\G_R\hdim M\}$ where 
\[b_M=\min\{n\mid \ext^{\gg 0}(M,B)=0\Rightarrow \ext^{>n}(M,B)=0\}.\]  Note that AB dimension satisfies the Auslander Buchsbaum formula.  Also, a ring is AB if and only if every module has finite AB dimension.  
\begin{lemma}
Taking $\B\sbe\md(R)$,  if $\ABdim M<\infty$ for all $M\in\DD(\M_\B)$, then $\M_\B$ is a thick subcategory of \MCM.
\end{lemma}
\begin{proof}
Suppose $\ABdim \DD(\M_\B)<\infty$.  First, we show that $\M_\B$ is contained in \MCM.  Take any $M\in\M_\B$.    There is an exact sequence $0\to M\to Y\to X\to 0$ with $\pd(Y)<\infty$ and $X\in \MCM$.  We claim that $X$ has AB dimension zero.  Suppose $\ext^{\gg 0}(X,Z)=0$.  Then $\ext^{\gg 0}(Y,Z)=0$ and since $\pd Y=\ABdim Y$, $\ext^{>\pd Y}(Y,Z)$ is zero.  Then we have $\ext^{\gg 0}(M,Z)=0$ and thus $\ext^{>b_M}(M,Z)=0$.   Therefore $\ext^i(X,Z)=0$ for all $i>\max\{\pd(Y),b_M\}+1$.  Since $R$ is Gorenstein, that means that $X$ has finite $\G_R$ dimension, and thus $X$ has finite AB dimension.  But since AB dimension satisfies the Auslander Buchsbaum formula, $\ABdim X$ must be zero.

Since $Y\in\DD(\M_\B)$, we have $X\in\DD(\M_B)$. So $\ext^{\gg 0}(X,B)=0$ for all $B\in\B$, and we have $\ext^{>0}(X,B)=0$ for all $B\in\B$. Hence $X$ is in $\M_\B$.  Therefore, $Y$ is also in $\M_\B$, which, by Lemma \ref{motiv}, means that $Y$ is projective and hence in \MCM, forcing $M$ to be in $\MCM$ as well.  

 Now to show that $\M_\B$ is thick in \MCM, it suffices to show that $\M_\B$ is closed under cokernels of surjections in \MCM.  So take $0\to L\to M\to N\to 0$ with $L,M,N\in \MCM$ and $L,M\in\M_\B$.  Then $N\in \DD(\M_\B)$ and so $\ext^{\gg 0}(N,B)=0$ for all $B\in\B$.  But then $N$ has finite AB dimension by assumption.  Since AB dimension satisfies the Auslander Buchsbaum formula,  $\ABdim N$ is zero.  So we have $\ext^{>0}(N,B)=0$ for all $B\in\B$, and hence, $N$ is in $\M_\B$.  

\end{proof}

Now let $d=\dim R$.  
\begin{theorem}\label{tea}
If $R$ is Gorenstein, then the following are equivalent.
\begin{enumerate}
\item\label{5} R is AB
\item\label{1} $\M_\B$ is a thick subcategory of $\MCM$ for all $\B\sbe\md(R)$
\item\label{2} $\MCM\cap\M_\B$ is thick in $\MCM$  for every $\B\sbe\md(R)$
\item\label{3} $\Ld_{\M_\B}$ gives a bijection between $\R(\M_\B)$ and $\Ga$ for every $\B\sbe\md(R)$
\item\label{4} For all $\B\sbe \md(R)$ and $M\in\M_\B$, $\Ga$ contains the function $f:\spec R\to \NN$ defined by the following: \[f(p)=\min\{n\mid \ext^{>n}(M_p,B_p)=0\quad\forall B\in\B\}\]
\end{enumerate}   
\end{theorem}
\begin{proof}
The previous lemma shows that (\ref{5}) implies (\ref{1}), and (\ref{1}) implies (\ref{2}) is trivial.  Assuming (\ref{2}), we will show (\ref{5}).  Suppose $\ext^{\gg 0}(M,B)=0$. Then $M$ is in $\DD(\M_B)$.  Letting $\dim R=d$, we have $\zz^d M\in\DD(\M_B)\cap \MCM$.  For some $n\ge d$ we have $\zz^n M\in\M_\B\cap \MCM$.  But then we have
\[0\to \zz^n M\to F_{n-1}\to\cdots\to F_d\to \zz^d M\to 0\]
where each $F_i$ is projective. By \ref{2}, $\zz^d M$ is in $\M_\B$.  So we have $\hdim_{C_B} M\le d$, and so $\ext^{>d}(M,B)=0$.

Theorem \ref{engine} shows that (\ref{1}) implies (\ref{3}). Lemma  \ref{7.1} shows that (\ref{3}) implies (\ref{4}).  Since $R$ is local, evaluating  $f$ at the maximal ideal shows that (\ref{4}) implies (\ref{5}).
\end{proof}  
\begin{corollary}
Set $r=d-depth M$.  If $R$ is AB and $\ext^{\gg 0}(M,B)=0$, then $\ext^r(M,B)\ne 0$.  Furthermore, if $\ext^{r}(M,B)=0$ or $\ext^i(M,B)\ne 0$ for $i>r$, then $\ext^j(M,B)\ne 0$ for arbitrarily large $j$.  
\end{corollary}
\begin{proof}
Suppose $R$ is AB.  Then (\ref{1}) holds and so $\M_B\hdim$ satisfies the Auslander Buchsbaum formula. If $\ext^{\gg 0}(M,B)=0$ then $r=\M_B\hdim M=\max\{n\mid \ext^n(M,B)\ne 0\}$.  The second statement is just the contrapositive of the first statement.  
\end{proof}
\begin{corollary}
If $R$ is Gorenstein and every resolving subcategory of $\MCM$ is thick, then $R$ is AB. 
\end{corollary}
\begin{proof}
The assumption implies (\ref{1}) in Theorem \ref{tea}.
\end{proof}

Thus if $\Ld$ in Theorem \ref{main3} is a bijection from $\Ss(R)\x \Ga$ to $\R$, then $R$ is AB. In \cite{Stevenson13}, Stevenson shows that when $R$ is a complete intersection, every resolving subcategory of $\MCM$ is closed under duals. The following gives  a necessary condition for this property.
\begin{corollary}
If $R$ is Gorenstein and every resolving subcategory of $\MCM$ is closed under duals, then $R$ is AB.  
\end{corollary}
\begin{proof}
Suppose every resolving subcategory of $\MCM$ is closed under duals. Let $\M\sbe \MCM$ be resolving. Let $-^*=\hom(-,R)$.  Then for every $M\in\X$, $(\zz M^*)^*$ is in $\M$.   By Lemma \ref{6.1}, $\M$ is thick.  The result follows from the previous corollary.
\end{proof}

\section*{Acknowledgements}
The author would like to thank his advisor, Hailong Dao, for his guidance, and also Ryo Takahashi for his insightful comments.  He would also like to thank the referee whose suggestions greatly improved this article.


%


\bibliographystyle{amsplain}
\bibliography{Bibliography}


\end{document}